\documentclass[11pt,reqno,twoside]{article}

\usepackage[hang]{footmisc}
\usepackage{lipsum}

\usepackage{fixltx2e} 

\usepackage{cmap} 

\usepackage[T1]{fontenc}
\usepackage[utf8]{inputenc}
\usepackage{graphicx}
\usepackage{placeins}
\usepackage{enumerate}
\usepackage{algcompatible}
\usepackage{subcaption}
\usepackage{verbatim}
\newcommand{\comments}[1]{}

\usepackage{soul}

\usepackage{setspace}

\let\counterwithin\relax  
\usepackage{lmodern} 
\usepackage[scale=0.88]{tgheros} 

\usepackage{bm} 

\usepackage{bbold}

\usepackage{amsmath,amsbsy,amsgen,amscd,amsthm,amsfonts,amssymb} 

\usepackage[centering,top=1.1in,bottom=1.4in,left=1in,right=1in]{geometry}

\usepackage{titling}
\usepackage{musicography}
\graphicspath{ {./images/} }

\usepackage[sf,bf,compact]{titlesec}

\usepackage{booktabs,longtable,tabu} 
\usepackage[font=small,margin=30pt,labelfont={sf,bf},labelsep={space}]{caption}

\usepackage[usenames,dvipsnames]{xcolor}

\definecolor{dark-gray}{gray}{0.3}
\definecolor{dkgray}{rgb}{.4,.4,.4}
\definecolor{dkblue}{rgb}{0,0,.5}
\definecolor{medblue}{rgb}{0,0,.75}
\definecolor{rust}{rgb}{0.5,0.1,0.1}

\usepackage{url}
\usepackage[colorlinks=true]{hyperref}
\hypersetup{linkcolor=dkblue}    
\hypersetup{citecolor=rust}      
\hypersetup{urlcolor=rust}     

\usepackage[final]{microtype} 

\newtheoremstyle{myThm} 
    {\topsep}                    
    {\topsep}                    
    {\itshape}                   
    {}                           
    {\sffamily\bfseries}                   
    {.}                          
    {.5em}                       
    {}  

\newtheoremstyle{myRem} 
    {\topsep}                    
    {\topsep}                    
    {}                   
    {}                           
    {\sffamily}                   
    {.}                          
    {.5em}                       
    {}  

\newtheoremstyle{myDef} 
    {\topsep}                    
    {\topsep}                    
    {}                   
    {}                           
    {\sffamily\bfseries}                   
    {.}                          
    {.5em}                       
    {}  

\theoremstyle{myThm}
\newtheorem{theorem}{Theorem}[section]
\newtheorem{lemma}[theorem]{Lemma}
\newtheorem{proposition}[theorem]{Proposition}
\newtheorem{corollary}[theorem]{Corollary}

\newtheorem{assumption}[theorem]{Assumption}

\theoremstyle{myRem}
\newtheorem{remarkx}[theorem]{Remark}
 \newenvironment{remark}
  {\pushQED{\qed}\remarkx}
  {\popQED\endremarkx}

\theoremstyle{myDef}
\newtheorem{definition}[theorem]{Definition}
 \newtheorem{example}[theorem]{Example}

\usepackage{fancyhdr}
\let\originalleft\left
\let\originalright\right
\renewcommand{\left}{\mathopen{}\mathclose\bgroup\originalleft}
\renewcommand{\right}{\aftergroup\egroup\originalright}

\usepackage{mathtools}
\mathtoolsset{centercolon}  

\definecolor{mygreen}{rgb}{0.1,0.75,0.2}

\newcommand{\nc}{\normalcolor}

\providecommand{\mathbbm}{\mathbb} 

\newcommand{\R}{\mathbbm{R}}

\newcommand{\N}{\mathbbm{N}}

\newcommand{\F}{\mathcal{F}}

\renewcommand{\phi}{\varphi}
\newcommand{\eps}{\varepsilon}

\usepackage[font = small, margin=30pt]{caption}

\usepackage[]{algorithm}
\usepackage{algpseudocode}
\usepackage{enumerate}
\usepackage{enumitem}

\usepackage{graphicx}

\usepackage{authblk}
\usepackage[square,numbers]{natbib}
\usepackage{chngcntr}
\usepackage{mathrsfs}
\counterwithin{table}{section}
\counterwithin{algorithm}{section}

\newcommand{\indicator}{{\bf{1}}}

\newcommand{\iid}{\stackrel{\text{i.i.d.}}{\sim}}

\newcommand{\E}{\mathbb{E}}

\renewcommand{\P}{\mathbb{P}}

\newcommand{\hatSigma}{\widehat{\Sigma}}

\newcommand{\hatk}{\hat{k}}

\newcommand{\hatC}{\widehat{C}}

\newcommand{\mcB}{\mathcal{B}}
\newcommand{\mcC}{\mathcal{C}}

\newcommand{\mcE}{\mathcal{E}}

\newcommand{\mcN}{\mathcal{N}}

\newcommand{\mcS}{\mathcal{S}}

\newcommand{\mcU}{\mathcal{U}}

\usepackage{multirow}

\newcommand{\CC}{\mathcal{C}}     

\newcommand{\inparen}[1]{\left(#1\right)}             

\usepackage[scr = esstix, cal = cm, frak=euler]{mathalfa}

\definecolor{mygreen}{rgb}{0.1,0.75,0.2}

\usepackage[scr = esstix, cal = cm, frak=euler]{mathalfa}

\title{Optimal Estimation of Structured Covariance Operators} 

\author{Omar Al-Ghattas, Jiaheng Chen, Daniel Sanz-Alonso and Nathan Waniorek}

\date{University of Chicago}

\vspace{.25in}

\makeatletter\@addtoreset{section}{part}\makeatother%
\numberwithin{equation}{section}

\newcommand{\upperRomannumeral}[1]{\uppercase\expandafter{\romannumeral#1}}

\renewcommand{\hat}{\widehat}

\begin{document}
\maketitle 


\begin{abstract}
This paper establishes optimal convergence rates for estimation of structured covariance operators of Gaussian processes. We study banded operators with kernels that decay rapidly off-the-diagonal and $L^q$-sparse operators with an unordered sparsity pattern. For these classes of operators, we find the minimax optimal rate of estimation in operator norm,
identifying the fundamental dimension-free quantities that determine the sample complexity. In addition, we prove that tapering and thresholding estimators attain the optimal rate. The proof of the upper bound for tapering estimators requires novel techniques to circumvent the issue that discretization of a banded operator does not result, in general, in a banded covariance matrix.
To derive lower bounds for banded and $L^q$-sparse classes, we introduce a general framework to lift theory from high-dimensional matrix estimation to the operator setting. Our work contributes to the growing literature on operator estimation and learning, building on ideas from high-dimensional statistics while also addressing new challenges that emerge in infinite dimension.
\end{abstract}

\section{Introduction}\label{sec:introduction}
 Across many problems in statistics, it is essential to constrain the model by imposing structural assumptions such as sparsity, smoothness, the manifold hypothesis, or group invariance. A vast body of work has demonstrated that these and other forms of structure facilitate inference of high-dimensional vectors, large matrices, graphs, networks, and functions \cite{buhlmann2011statistics,tsybakov2008introduction,wainwright2019high,gine2021mathematical,garcia2022mathematical,newman2018networks,eaton1989group,kondor2008group}. This paper sets forth the study of operator estimation and its fundamental limits under natural structural assumptions. 
We consider two classes of covariance operators: banded integral operators with kernels that decay rapidly off-the-diagonal, and a more flexible family of $L^q$-sparse operators where the kernel need not concentrate around its diagonal. For both classes, we establish optimal convergence rates using a general framework to lift theory from high-dimensional matrix estimation to the operator setting. In so doing, we identify the dimension-free quantities that determine the sample complexity.  Additionally, we show that tapering and thresholding estimators achieve the minimax optimal rate.

 Our motivation to study covariance operator estimation stems from the growing interest in data-driven regularizers for inverse problems in function space \cite{stuart2010inverse}.  In imaging applications, unlabeled data are routinely used to learn Tikhonov regularizers and prior covariance models \cite{arridge2019solving}.
Similarly, operational algorithms for numerical weather prediction rely on an ensemble of forecasts to estimate a background prior covariance \cite{evensen2009data}.  In these applications and many others, the data used to specify the prior covariance represent finely discretized functions. As data resolution continues to improve, we wish to understand the fundamental dimension-free, discretization-independent quantities that determine the difficulty of estimating the prior covariance.  Relatedly, operator learning, i.e. the task of recovering an operator from pairs of inputs and outputs or from trajectory data \cite{kovachki2024data,kim2024bounding,liu2024deep,stepaniants2023learning,de2023convergence,jin2022minimax,mollenhauer2022learning,lu2021learning}, has also received increased attention motivated by recent machine learning techniques to solve partial differential equations, see e.g. \cite{han2018solving,khoo2021solving,raissi2019physics,lifourier}. We emphasize the importance of analyzing the operator version of the covariance estimation problem, since many scientific problems are most naturally formulated in function space. While computations will inevitably need to be carried out with discretizations, designing and analyzing algorithms and estimation procedures in the function space setting often yields scalable and discretization-robust theory and methods \cite{cotter2013mcmc,deuflhard2011newton,trillos2017consistency,sanz2024analysis}.

 Banded structure in operators arises naturally in time series analyses and spatial datasets as a consequence of decay of correlations in time or space. Our theory shows that tapering estimators, akin to popular covariance localization techniques in the geophysical sciences \cite{furrer2007estimation,houtekamer2016review} that often rely on the Gaspari-Cohn tapering function \cite{gaspari1999construction}, are minimax optimal. In addition, tapering estimators are computationally appealing, as they only estimate the kernel around its diagonal, thus reducing computational and memory costs. $L^q$-sparsity structure in operators is a weaker requirement than bandedness, and, not surprisingly, $L^q$-sparse operators are more challenging to estimate.
 While tapering estimators fail in general over the class of $L^q$-sparse operators, we show that thresholding estimators achieve the minimax rate.   A caveat is that thresholding estimators require to pre-compute the full sample covariance and to threshold it point-wise at higher computational cost.

\subsection{Related work}\label{ssec:relatedwork}
An exhaustive review of the vast literature on covariance estimation is beyond the scope of this work, and we refer to \cite{cai2016estimating} for a survey article.
Here, we provide a focused overview of high-dimensional structural assumptions that inspire our infinite-dimensional theory, and a brief summary of existing results in infinite dimension.   
\subsubsection{Covariance matrix estimation: Banded assumption}
The seminal work \cite{bickel2008regularized} considered covariance estimation over the class of banded matrices
\begin{align}\label{eq:bandedclassfinitedimension}
\begin{split}
    \mcB(p, \alpha) = \Biggl\{
        \Sigma \in \mcS_+^p: 
            \| \Sigma \| \le M_0, 
            ~ \max_{i \le p} \sum_{\{j : |i-j| > m\}} |\Sigma_{ij}| \le M m^{- \alpha}, \, \forall \ m \in \N \Biggr\},
\end{split}            
\end{align}
where $\mcS_+^p$ denotes the set of $p\times p$ symmetric positive-definite matrices, and $\| \cdot \|$ denotes the operator norm. The banded estimator studied in \cite{bickel2008regularized} was shown to be suboptimal in \cite{cai2010optimal}, which proposed the tapering estimator $ {\hatSigma}_\kappa:= \hatSigma \circ T_\kappa,$ where $\hatSigma$ is the sample covariance, $\circ$ denotes the Schur product,  $\kappa >0$ is the tapering radius,  and $T_\kappa = (t_{ij})_{1 \le i, j \le p} $ is the tapering matrix with entries
\begin{align}\label{eq:taperingmatrix}
\begin{split}
    t_{ij}
    &=   \frac{(2\kappa-|i-j|)_+ - (\kappa - |i-j|)_+}{\kappa} 
    = \begin{cases}
        1 \quad & \text{if} \quad |i-j| \le \kappa,\\
        2 - \frac{|i-j|}{\kappa} \quad  &\text{if} \quad \kappa < |i-j| \le 2\kappa,\\
        0 \quad  &\text{if} \quad |i-j|> 2\kappa,
    \end{cases}
\end{split}    
\end{align} 
where $(x)_+ := \max \{x, 0 \}$. The paper \cite{cai2010optimal} showed that if the sample size $N$ satisfies $\frac{\log p}{N}\lesssim 1$ and $\kappa = \min \bigl\{ \bigl\lceil N^{\frac{1}{2 \alpha + 1}} \bigr\rceil, p \bigr\}$, then with high probability it holds that 
    \begin{equation}\label{eq:bandederrorfinited}
        \| \hatSigma_\kappa  - \Sigma \|^2 \lesssim \min \left \{ 
            N^{-\frac{2 \alpha}{2 \alpha + 1}} + \frac{\log p}{N}, \frac{p}{N}
        \right \}.
    \end{equation}
A matching minimax lower bound was proved using Assouad's lemma and Le Cam's method.

\subsubsection{Covariance matrix estimation: Sparsity assumption}
The banded class \eqref{eq:bandedclassfinitedimension} presupposes a natural ordering of the variables, so that entries $\Sigma_{ij}$ are small whenever $|i-j|$ is large. The approximate $\ell^q$-sparsity class $(0\le q\le 1)$  \begin{align}\label{eq:SparseMatrixClass}
     \mcU(p,R_q) = 
     \Biggl\{
     \Sigma \in \mcS_+^p: 
     \max_{i \le p} \Sigma_{ii} \le M, ~ \max_{i \le p} \sum_{j=1}^p |\Sigma_{ij}|^q \le R_q^q\Biggr\}
\end{align}
dispenses of the requirement that variables be ordered, while retaining the model assumption of row-wise approximate sparsity. In the class $\mcU(p,R_q),$ banded and tapering estimators perform poorly in general, and thresholding estimators, as  introduced in \cite{bickel2008covariance} and further studied in \cite{rothman2009generalized, cai2012optimal, cai2011adaptive,cai2016estimating}, are favored. The idea is to threshold the entries of the sample covariance that are below a pre-specified value $\rho$, i.e. set\footnote{We  abuse notation and denote tapering and thresholding estimators by $\hatSigma_\kappa$ and $\hatSigma_\rho.$ This will cause no confusion since we consistently denote by $\kappa$ and $\rho$ tapering and thresholding radii, respectively.} $ \hatSigma_{\rho} = \bigl(\hatSigma_{ij} \indicator \{|\hatSigma_{ij}| \ge \rho\}\bigr)_{1 \le i, j \le p}$. 
The paper \cite{bickel2008covariance} showed that if $\frac{\log p}{N}=o(1)$ and $\rho=M^{\prime}\sqrt{\frac{\log p}{N}}$ for sufficiently large $M^{\prime}$, then
\begin{equation}\label{eq:thresholderrorfinited}
\|\hatSigma_{\rho}-\Sigma\|=O_{\mathbb{P}}\bigg(R_q^q \Bigl(\frac{\log p}{N}\Bigr)^{(1-q)/2}\bigg).
\end{equation}
As shown in \cite[Theorem 6.27]{wainwright2019high}, the choice of thresholding level $\rho\asymp \sqrt{\frac{\log p}{N}}$ ensures an entry-wise control on the deviation of the sample covariance matrix from its expectation with high probability. Further, \cite{cai2012optimal} established the optimality of thresholding estimators by deriving a sharp minimax lower bound through a general ``two-directional'' technique.

\subsubsection{Covariance operator estimation: Unstructured setting}
Let $u: D \to \R$ be a centered  square-integrable Gaussian process on a bounded domain $D \subset \R^d$. 
The covariance function (kernel) $k:D\times D \to \R$ and operator $\CC: L^2(D) \to L^2(D)$ of $u$ satisfy that, for any $x,y \in D$ and $\psi \in L^2(D)$, 
\begin{align*}
    k(x,y) := \E \bigl[u(x)u(y)\bigr], 
    \qquad
    (\CC\psi)(\cdot) := \int_D k(\cdot, y)\psi(y) \, dy.
\end{align*}
Given data $\{ u_n\}_{n=1}^N$ comprising $N$ independent copies of $u,$
the sample covariance function $\hatk$ and sample covariance operator $\widehat{\CC}$ are defined  by 
\begin{align*}
    \hatk(x,y) := \frac{1}{N} \sum_{n=1}^N u_n(x) u_n(y),
    \qquad 
    (\widehat{\CC}\psi)(\cdot) := \int_D \hatk(\cdot, y)\psi(y) \, dy,
\end{align*}
where here and throughout this paper we will tacitly assume that point-wise evaluations of $u$ are almost surely well-defined Lebesgue almost-everywhere. 
The work \cite{koltchinskii2017concentration} shows that, for any sample size $N,$
    \begin{align}\label{eq:Koltchinksiibound}
         \frac{\mathbb{E} \| \widehat{\mathcal{C}} - \mathcal{C} \|}{\|\mathcal{C} \| }  \asymp   \sqrt{\frac{r(\mathcal{C})}{N} } \lor  \frac{r(\mathcal{C})}{N}, 
         \qquad  
         r(\mathcal{C}) := \frac{\mathrm{Tr}(\mathcal{C})}{\| \mathcal{C}\|},
    \end{align}
where $r(\CC)$ is called the \emph{effective dimension}, and $\mathrm{Tr}(\mathcal{C})$ denotes the trace of $\CC$. It is also known that the term $\sqrt{r(\CC)/N}$ appears in the minimax lower bound for covariance matrix estimation \cite[Theorem 2]{lounici2014high}. 
More recently, \cite[Theorem 2.3]{han2022exact} establishes an exact upper bound with optimal constants for the relative error
in \eqref{eq:Koltchinksiibound} using Gaussian comparison techniques.

\subsubsection{Covariance operator estimation: Banded and sparsity assumptions}
To our knowledge, this paper establishes the first upper and lower bounds for estimation of banded covariance operators. An upper bound for thresholding estimators under $L^q$-sparsity was established in \cite[Theorem 2.2]{al2025covariancea}; in this paper, we establish the minimax lower bound.  

\subsubsection{Other related topics}\label{ssec:other} 
  In this work, we assume access to full sample paths from the process, in contrast to the partially observed setting in functional data analysis, where paths are noisily observed at finitely many locations. With the notable exception of \cite{mohammadi2024functional}, it is standard in the partially observed setting to first reconstruct the paths using nonparametric smoothers \cite{ramsay2002applied, yao2005functional, zhang2016sparse}, which require additional regularity assumptions that affect the final bound. 
 Our bounds rely instead on fundamental dimension-free quantities that capture the complexity of the underlying process. We refer to \cite[Remark 2.5]{al2025covarianceb} for further discussion.  Another line of research that dates back to the celebrated Marchenko–Pastur law \cite{marchenko1967distribution} has focused on the spectral properties of the sample covariance \cite{johnstone2009consistency,ledoit2011eigenvectors,bloemendal2016principal} and related estimation techniques, e.g. eigenvalue shrinkage estimators \cite{stein1986lectures,ledoit2012nonlinear,donoho2018optimal,bun2017cleaning}. More recently, there is a growing literature in estimating smooth functionals of covariance operators \cite{koltchinskii2021asymptotically,koltchinskii2018asymptotic,koltchinskii2021estimation,koltchinskii2024estimation}. Further related topics and open directions will be discussed in the conclusions section.

\subsection{Main contributions and outline}
\begin{itemize}
    \item Section \ref{sec:banded} investigates estimation of banded covariance operators. The main results, Theorems \ref{thm:bandable_upper_bound} and \ref{thm:bandable}, show an upper bound for tapering estimators and a matching minimax lower bound.  The proof of the upper bound in Theorem \ref{thm:bandable_upper_bound} requires novel ideas to address new challenges that emerge in infinite dimension.  To prove Theorem \ref{thm:bandable}, we introduce in Proposition \ref{prop:reduction} a lower bound reduction framework to lift theory from high to infinite dimension. We further discuss the dimension-free quantities that determine the sample complexity,  relating them to the correlation lengthscale in Corollary \ref{cor:SmallLengthscaleBanded}.
    \item Section \ref{sec:sparse} considers estimation of $L^q$-sparse operators mirroring the presentation in Section \ref{sec:banded}. The main result, Theorem \ref{thm:approximate sparsity lower bound}, shows a minimax lower bound that matches an existing upper bound for thresholding estimators, established in \cite{al2025covariancea} and reviewed in Theorem \ref{thm:sparse_upper_bound}. The proof of Theorem \ref{thm:approximate sparsity lower bound} builds again on our lower bound reduction framework. We discuss the key dimension-free quantities that determine the sample complexity,  relating them to the correlation lengthscale in Corollary \ref{corr:LengthscaleApproxSparse}. 
    \item Section \ref{sec:Numerics} compares the numerical performance of tapering and thresholding estimators in problems with ordered and unordered sparsity patterns. 
    \item  Section \ref{sec:conclusions} closes with open questions that stem from this work. 
\end{itemize}

\subsection{Notions of dimension}\label{ssec:dimension}
In this paper, we study the estimation of the covariance operator $\mathcal{C}: L^2(D) \to L^2(D)$ of a Gaussian process $u:D \to \R$ defined on the domain $D = [0,1]^d$ based on $N$ independent copies $u_1, \ldots, u_N.$ It is important to clearly differentiate between the dimension of the parameter space, the dimension $d$ of the domain $D$ where the process $u$ is defined, and the dimension of the data. We use the following conventions and terminology:
\begin{itemize}
    \item \emph{Parameter space dimension:} 
    We will consider infinite-dimensional parameter spaces of banded and $L^q$-sparse operators on $L^2(D)$.  
    Our bounds are \emph{dimension-free} in that they do not depend on the dimension of the parameter space ---which is infinite--- but only on fundamental finite quantities such as the \emph{effective dimension} $r(\CC) = \mathrm{Tr}(\CC)/\|\CC\|$ of the true covariance operator, the off-diagonal decay of its kernel, and the expected supremum of the process $u$. Dimension-free bounds are essential to obtain \emph{discretization-independent} control of the estimation error; see e.g. \cite{sanz2024analysis}.  
    \item \emph{Physical dimension:} We use this term to refer to the dimension $d$ of the domain $D = [0,1]^d$ where the Gaussian process $u: D \to \R$ is defined. Since in most scientific applications the physical dimension is small (typically $d\leq 3)$, we treat it as a constant in our analysis. We expect, however, our bounds to worsen in the large $d$ asymptotic.
    \item \emph{Data dimension:} In this work, the data $u_1, \ldots, u_N$ represent $N$ independent copies of a Gaussian process. Each data point belongs to the infinite-dimensional space of square-integrable functions on $D=[0,1]^d$. Subsection \ref{ssec:other} compares this setting with the partially observed setting in functional data analysis, where the data are finite-dimensional vectors. 
\end{itemize}


\subsection{Notation}  
For positive sequences $\{a_n\}, \{b_n\}$, we write $a_n \lesssim b_n$ to denote that, for some constant $c >0$, $a_n \le c b_n.$ If both $a_n \lesssim b_n$ and $b_n \lesssim a_n$ hold, we write $a_n \asymp b_n.$  
For an operator $\CC,$ we denote its trace by $\text{Tr}(\CC)$ and its operator norm by $\|\CC\|.$ 
For a vector $x\in \R^d$, we use $\|x\|$ to denote its Euclidean norm.

\section{Estimating banded covariance operators} \label{sec:banded}

\subsection{Upper bound}
 Motivated by the covariance matrix class \eqref{eq:bandedclassfinitedimension}, we work under the following assumption. Recall that $r(\CC) = \mathrm{Tr}(\CC)/\|\CC\|$ denotes the effective dimension.
\begin{assumption}\label{assumption:main1}
The data $\{u_n\}_{n=1}^N$ consists of $N$ independent copies of a real-valued, centered, square-integrable Gaussian process $u$ on $D=[0,1]^d$ with covariance function $k:D \times D \to \R$ and trace-class covariance operator $\CC: L^2(D) \to L^2(D),$ denoted $u \sim \mathrm{GP}(0,\CC).$  Moreover:
    \begin{enumerate}[label=(\roman*)]
        \item $\sup_{x\in D} k(x,x)\lesssim \mathrm{Tr}(\CC)$. \label{assumption:magnitude}
        \item There exists a positive and decreasing sequence $\{\nu_m\}_{m=1}^{\infty}$ with $\nu_1=1$ such that 
        \begin{equation}
        \sup_{x\in D} \int_{\{y: \|x-y\| > m r(\CC)^{-1/d}\}} |k(x,y)| dy \lesssim   \|\CC\|\nu_m \nc, \qquad \forall \ m \in \N. \tag*{\qed}
        \end{equation}\label{assumption:sparsity}
    \end{enumerate}
\end{assumption}

\begin{remark}
Notice that it always holds that $\mathrm{Tr}(\CC) \le \sup_{x\in D}k(x,x),$ since
\[
\mathrm{Tr}(\CC)= \E_{u \sim \mathrm{GP}(0,\CC)} \|u\|_{L^2(D)}^2  = 
\int_{D}k(x,x)dx\le \Big(\sup_{x\in D}k(x,x)\Big)\mathrm{Vol}(D)=\sup_{x\in D}k(x,x).
\]
 Assumption \ref{assumption:main1} (i) is satisfied when the reverse inequality also holds, in which case the marginal variance varies moderately across the domain $D$. Our analysis will show that this requirement plays a crucial role in linking global covariance estimation to local estimates.
 The condition in  Assumption \ref{assumption:main1} (ii) resembles that in the finite-dimensional counterpart \eqref{eq:bandedclassfinitedimension}, which restricts attention to the case $\nu_m = m^{-\alpha}$. As discussed in Section \ref{subsection:lengthscale-banded}, the quantity $r(\CC)^{-1/d}$ in our assumption can be viewed as a natural \emph{lengthscale} in the interesting regime where $r(\CC)\gg 1.$ In that setting, the sequence $\{\nu_m\}$ controls the tail decay of the covariance function. 
\end{remark}

For a tunable parameter $\kappa>0,$ we define the tapering estimator as
\[
\hat{k}_{\kappa}(x,y):=\hat{k}(x,y)f_{\kappa}(x,y),\qquad (\hat{\CC}_{\kappa}\psi)(\cdot) := \int_D \hat{k}_{\kappa}(\cdot, y)\psi(y) \, dy, \ \ \psi\in L^2(D), 
\]
where the tapering function $f_\kappa$ is defined as
\begin{align}\label{eq:tapering_func}
f_{\kappa}(x,y):=\prod_{i=1}^{d}\min\left\{\frac{(2\kappa-|x_i-y_i|)_{+}}{\kappa},1\right\}.
\end{align}

\begin{remark}
    When $d=1$, the tapering function in \eqref{eq:tapering_func} becomes 
    \[
    f_{\kappa}(x,y)=\min\left\{\frac{(2\kappa-|x-y|)_{+}}{\kappa},1\right\},
    \]
    which is identical to the tapering function used in the matrix setting \cite{cai2010optimal}. Therefore, $f_{\kappa}(x,y)$ in \eqref{eq:tapering_func} can be seen as a generalization of the tapering function in \cite{cai2010optimal} to high physical dimension $d > 1$.  
\end{remark}

We next define two important quantities, $m_*$ and $\varepsilon_*,$ that will be respectively used to specify the tapering radius $\kappa$ and to control the estimation error.
\begin{definition}\label{def:1}
Define the pair $m_{*}=m_*(\{\nu_m\},N,d)$ and $\varepsilon_*=\varepsilon_*(\{\nu_m\},N,d)$ as
\begin{align*}
    m_* := \min \bigg\{ m \in \N: \nu_{m} \le \sqrt{\frac{m^d}{N}} \bigg\},
    \qquad 
    \varepsilon_*:=\max_{m\in \N} \bigg\{ \nu_m \land \sqrt{\frac{m^d}{N}}\bigg\}.
\end{align*}
\end{definition}

The quantity $m_*$ is essentially the solution to the maximization problem in the
definition of $\varepsilon_*$, which can be conceptualized as the correct truncation order. This relationship is formalized in Lemma~\ref{lemma:basic_relation}. Similar quantities arise in various statistical problems, e.g. estimation and testing in sequence models \cite{massart2001gaussian, baraud2002non}. In nonparametric regression and density estimation, the rate for Sobolev spaces with smoothness $\alpha$ is $\varepsilon_*\asymp N^{-\frac{\alpha}{2\alpha+d}}$ since $\nu_m\asymp m^{-\alpha}$, corresponding to the decay rate of the function's coefficients when decomposed along some orthonormal basis.

Notice that if $m_*=1$, then $N=1$ by Definition \ref{def:1}, which is the trivial case. For simplicity, we assume $m_*\ge 2$ throughout the paper. 
We are now ready to state our first main result, which provides an upper bound on the estimation error. 

\begin{theorem}\label{thm:bandable_upper_bound}
Suppose that Assumption \ref{assumption:main1} holds and set $\kappa:=m_{*} r(\CC)^{-1/d}.$ Then,
\[
\frac{\E \|\hat{\CC}_{\kappa}-\CC\|}{\|\CC\|}\lesssim \varepsilon_* +\biggl(\sqrt{\frac{\log r(\CC)}{N}}\lor \frac{\log r(\CC)}{N}\biggr).
\]
\end{theorem}

\begin{example}
    If $\nu_m=m^{-\alpha}$, then $m_{*}\asymp N^{\frac{1}{2\alpha+d}}$ and $\varepsilon_*\asymp N^{-\frac{\alpha}{2\alpha+d}}$; if $\nu_m=e^{-m^t}$, then $m_*\asymp (\log N)^{1/t}$ and $\varepsilon_*\asymp \sqrt{(\log N)^{d/t} / N} $. 
\end{example}

 To our knowledge, Theorem \ref{thm:bandable_upper_bound} gives the first upper bound in the literature for tapering estimators under the natural bandedness Assumption \ref{assumption:main1}. 
We establish a matching minimax lower bound in  Theorem~\ref{thm:bandable}, thus showing that tapering estimators attain the minimax optimal rate.

The upper bound in Theorem~\ref{thm:bandable_upper_bound} is analogous to the matrix result in \cite{cai2010optimal}; see also  \eqref{eq:bandederrorfinited}. 
In particular, our choice of tapering radius $\kappa=m_{*} r(\CC)^{-1/d}$  depends on $\{\nu_m\}$ and on $r(\CC);$ this is  analogous to the dependence of $\kappa$ on $\alpha$ and $p$ in \cite{cai2010optimal}. 
However, several key differences arise in our setting. First, our bound is \emph{dimension-free} in the sense of Subsection \ref{ssec:dimension}, and is stated in terms of \emph{relative error}. Second, discretizing a covariance operator that satisfies Assumption \ref{assumption:main1} yields a banded covariance matrix only in physical dimension $d=1.$ As a result, the analysis of tapering estimators for banded covariance matrices in \cite{cai2010optimal} does not directly extend to our setting.  

 The proof of Theorem~\ref{thm:bandable_upper_bound} contains several novel ideas that we now outline. To streamline the flow of the paper, we defer the full proof to  Appendix~\ref{app:upperbandable}.
 The argument begins with a bias-variance trade-off analysis. The main technical challenge lies in the variance analysis: specifically, reducing global estimates of the variance to \emph{dimension-free} local estimates. A crucial observation is the following representation of the tapering function:
\[
f_{\kappa}(x,y)=\kappa^{-d}\sum_{\sigma\in \{1,2\}^d}(-1)^{\sum_{i=1}^{d}\sigma_i}\prod_{i=1}^{d} (\sigma_i\kappa-|x_i-y_i|)_{+},
\]
shown in Lemma \ref{lemma:f_kappa}. This expression reveals that the tapering function is a sum of separable product terms of the form $\prod_{i=1}^{d} (\sigma_i\kappa-|x_i-y_i|)_{+}$. Moreover, each such product can be expressed as a mixture of indicator functions:
\begin{align*}
\prod_{i=1}^{d} (\sigma_i\kappa-|x_i-y_i|)_{+}=\int_{\bar{D}}\mathbf{1}\bigl\{x,y\in T(\theta) \bigr\} \, d\theta,
\end{align*}
where $T(\theta):=\bigotimes_{i=1}^{d}[\theta_i-\frac{\sigma_i\kappa}{2},\theta_i+\frac{\sigma_i\kappa}{2}]$ is a small box centered at $\theta$, and $\bar{D}$ denotes a slightly enlarged version of the domain $D=[0,1]^d$. This representation enables a reduction from global to local variance analysis.
 To transition from the continuous uniform mixture distribution over $\theta$ to 
 disjoint local boxes, we introduce a measure $Q$ under which the centers $\theta$ are well separated (see Lemma \ref{lemma:Q_lowerbound}). We then apply a \emph{change of measure} argument to complete the reduction step. For the analysis of the local variance, we apply dimension-free concentration bounds for sample covariance operators from \cite{koltchinskii2017concentration}.  To lift the resulting local estimates to local bounds we use Lemma~\ref{lemma:disjoint}, which states that the operator norm of a sum of local covariance operators with disjoint supports equals the maximum of their individual operator norms. Finally, the parameter $\kappa$ is chosen to optimally balance the bias and variance terms.

\begin{remark}
        In the covariance matrix literature, \cite{cai2012adaptive,cai2016minimax} introduced a block-thresholding approach that adapts to $\alpha.$ In our infinite-dimensional setting, adaptation to $\alpha,$ or, more generally, to the decay rate of the sequence $\{\nu_m\},$ is an interesting direction for future work. The papers \cite[Lemma 2.1]{mendelson2020robust} (see also \cite[Lemma B.8]{al2024ensemble}) study estimation of $r(\CC)^{-1/d}$ up to a multiplicative constant. 
\end{remark}

\subsection{Lower bound}\label{ssec:lowerboundbanded}

In this subsection, we formulate and prove a matching lower bound for the estimation of banded covariance operators. Consider, as in Assumption \ref{assumption:main1}, the banded class
\begin{align}\label{eq:continuous banded class}
\F(r,\{\nu_m\})&:=\bigg\{\CC: \sup_{x\in D}k(x,x)\lesssim \mathrm{Tr}(\CC), r(\CC)\le r, \nonumber\\ 
&\qquad \qquad \sup_{x\in D} \int_{\{y: \|x-y\| > m r(\CC)^{-1/d}\}} |k(x,y)| dy \lesssim   \|\CC\| \nu_m, \forall \, m \in \N \bigg\}.
\end{align}
The following theorem establishes a minimax lower bound over this class.

\begin{theorem}\label{thm:bandable}
Suppose $N > \log r>0$. The minimax risk for estimating the covariance operator over $\F(r,\{\nu_m\})$ under the operator norm satisfies
\[
\inf_{\hat{\CC}}\sup_{\CC\in \F(r,\{\nu_m\})}\frac{\mathbb{E} \| \widehat{\mathcal{C}} - \mathcal{C} \|}{\|\mathcal{C} \| } \gtrsim \min\left\{ \varepsilon_*+\sqrt{\frac{\log r}{N}},\sqrt{\frac{r}{N}}\right\}. \]
\end{theorem}

The next proposition provides a general technique to reduce the \textit{infinite-dimensional} operator estimation problem to a \textit{finite-dimensional} matrix estimation problem. This proposition, proved in Appendix \ref{app:lowerbandable}, will be used to establish lower bounds for both banded and sparse covariance operators. 

\begin{proposition}[Lower bound reduction]\label{prop:reduction}
Let $I_1,I_2,\ldots, I_M$ be a uniform partition of $D=[0,1]^d$ with $\mathrm{Vol}(I_i)=\frac{1}{M}$. Let $\F\subseteq \R^{M\times M}$ be a subset of positive semi-definite matrices. For every $\Sigma\in \F$, define
\begin{align*}
    k_{\Sigma}(x,y) :=\sum_{i,j=1}^{M} \Sigma_{ij}\mathbf{1}\{x\in I_i\}\mathbf{1}\{y\in I_j\},
    \qquad 
    (\CC_{\Sigma}\psi)(\cdot) := \int_D k_{\Sigma}(\cdot, y)\psi(y) \, dy.
\end{align*}
Then,
\begin{enumerate}[label=(\alph*)]
\item $\CC_{\Sigma}: L^2(D) \to L^2(D)$ is positive semi-definite and trace-class.
\item $\|\CC_{\Sigma}\|=\frac{1}{M}\|\Sigma\|$. \label{eq:reductionb}
\item Let $\F^{*}=\{\CC_{\Sigma}:\Sigma\in \F\}$, $u_1(\cdot),u_2(\cdot),\ldots, u_N(\cdot)\iid \mathrm{GP}(0,\CC),$ and $X_1,X_2, \ldots, X_N\iid \mcN(0,\Sigma)$. 
Then, the following holds:
\begin{align}\label{eq:reduction1}
\inf_{\hat{\CC}}\sup_{\CC\in \F^{*}} \E \|\hat{\CC}-\CC\|\ge \frac{1}{M}\inf_{\hat{\Sigma}}\sup_{\Sigma\in \F} \E \|\hat{\Sigma}-\Sigma\|
\end{align}
and
\begin{align}\label{eq:reduction2}
\inf_{\hat{\CC}}\sup_{\CC\in \F^{*}} \frac{\mathbb{E} \| \widehat{\mathcal{C}} - \mathcal{C} \|}{\|\mathcal{C} \| }\ge \inf_{\hat{\Sigma}}\sup_{\Sigma\in \F} \frac{\mathbb{E} \| \widehat{\Sigma} - \Sigma \|}{\|\Sigma \| },
\end{align}
where $\inf_{\hat{\CC}}$ is taken over kernel
integral operators $\widehat{\CC}$ whose kernel $\hat{k}$ is a measurable function of $\{u_n(\cdot)\}_{n=1}^N$ and $\inf_{\hat{\Sigma}}$ is taken over measurable functions $\hat{\Sigma}$  of $\{X_n\}_{n=1}^N$.
\end{enumerate}
\end{proposition}
 Proposition \ref{prop:reduction} formalizes the intuition that the infinite-dimensional covariance operator estimation problem is at least as hard as the finite-dimensional covariance matrix estimation problem. This proposition facilitates using existing information-theoretic lower bounds for high-dimensional covariance matrix estimation to prove lower bounds for operator estimation with relatively few modifications. In this paper, we illustrate this claim by obtaining lower bounds for banded and sparse covariance operators building on the covariance matrix lower bounds in \cite{cai2010optimal} and \cite{cai2012optimal}, respectively. We expect that this approach could also be used to prove sharp lower bounds for other structured covariance operator estimation problems, provided that the fundamental dimension-free quantities are correctly identified.

\begin{proof}[Proof of Theorem \ref{thm:bandable}]
We will follow the high-dimensional analysis in \cite{cai2010optimal}  to construct three subclasses $\mathcal{F}_i(r, \{\nu_m\}) \subseteq \mathcal{F}(r, \{\nu_m \})$ for $i = 1, 2, 3$, each tailored to a different regime of the parameters $(N, r, m_*)$. For each regime, we use Proposition~\ref{prop:reduction} to verify the inclusion $\mathcal{F}_i \subseteq \mathcal{F}$ and to 
derive a lower bound over $\mathcal{F}_i.$ 
The claimed result follows since the supremum over a subset provides a valid lower bound.

\paragraph{Case 1 ($N > \log r$).}
Suppose that $r > 1$ is an integer (otherwise, replace $r$ with $\lceil r \rceil$), and let $\{I_1, \ldots, I_r\}$ be a uniform partition of $D=[0,1]^d$ such that $\mathrm{Vol}(I_i) = 1/r$ and $\mathrm{diam}(I_i) \asymp r^{-1/d}$. Define
\[
\mathcal{F}_{0}:=\Bigg\{\Sigma_\ell: \Sigma_\ell=w I_r-w\left(\sqrt{\frac{\tau}{N} \log r}\,\, \mathbf{1}\{i=j=\ell\}\right)_{r \times r}, 0 \leq \ell \leq r\Bigg\}\subseteq \R^{r\times r},
\]
where $w > 0$ is a large constant and $\tau > 0$ is a sufficiently small constant, and define
\[
\mathcal{F}_1(r, \{\nu_m\}) := \left\{ \mathcal{C} : k(x,y) = \sum_{i,j=1}^r \Sigma_{ij} \mathbf{1}\{x \in I_i\} \mathbf{1}\{y \in I_j\}, \ \Sigma \in \mathcal{F}_0 \right\}.
\]
Lemmas~\ref{lemma:banding lower bound inclusion 1} and \ref{lemma:lower_bound1} then establish, using Proposition~\ref{prop:reduction}, that $\mathcal{F}_1(r, \{\nu_m\}) \subseteq \mathcal{F}(r, \{\nu_m\})$ and 
\begin{align*}
\inf_{\widehat{\mathcal{C}}} \sup_{\mathcal{C} \in \mathcal{F}_1(r, \{\nu_m\})} \frac{\mathbb{E} \| \widehat{\mathcal{C}} - \mathcal{C} \|}{\|\mathcal{C} \|} &\gtrsim \sqrt{\frac{\log r}{N}}.
\end{align*}

\paragraph{Case 2 ($N > \log r$, $r > m_*^d$).}
Let $S^d = r$, $K = (m_* - 1)/(2\sqrt{d})$, $\gamma_N = K^d$, and $h_N = K^{-d} \sqrt{m_*^d / N}$. For each $\theta = (\theta_1, \theta_2,\dots, \theta_{\gamma_N}) \in \{0,1\}^{\gamma_N}$, define
\[
k_\theta(x,y) := \sum_{\ell=1}^{S^d} \mathbf{1}\{x, y \in I_\ell\} + \tau h_N \sum_{\ell=1}^{\gamma_N} \theta_\ell \left( \mathbf{1}\{x \in I_\ell\} \mathbf{1}\{y \in T_{x,2K}\} + \mathbf{1}\{y \in I_\ell\} \mathbf{1}\{x \in T_{y,2K}\} \right),
\]
where $\tau \in (0, 4^{-d-1})$ is a small constant and the domains $\{I_{\ell}\}_{\ell=1}^{S^d}$ and $\{I_{\ell}\}_{\ell=1}^{\gamma_N}$ are given by
\[
\{I_{\ell}\}_{\ell=1}^{S^d}=\left\{\bigotimes_{\ell=1}^{d} \left[\frac{i_{\ell}}{S},\frac{i_{\ell}+1}{S}\right]: i_{\ell}\in \{0,1,\ldots,S-1\}\right\},
\]
\[
\{I_{\ell}\}_{\ell=1}^{\gamma_N}=\left\{\bigotimes_{\ell=1}^{d} \left[\frac{i_{\ell}}{S},\frac{i_{\ell}+1}{S}\right]: i_{\ell}\in \{0,1,\ldots,K-1\}\right\}.
\]
For $x \in \left[0,\frac{K}{S}\right]^d$, $T_{x,2K}:=\bigotimes_{i=1}^d \left[\frac{1+\lceil S x_i \rceil}{S},\frac{2K}{S}\right]$. Note that the assumption $r>m_*^d$ implies $S>2K$. Since $k_{\theta}(x,y)$ is a constant on each $I_{\ell}$ $(1\le \ell\le S^d)$, it admits the form
\[
k_{\theta}(x,y)=\sum_{\ell,\ell^{\prime}=1}^{S^d}\Sigma^{(\theta)}_{\ell \ell^{\prime}}\mathbf{1}\{x\in I_{\ell}\}\mathbf{1}\{y\in I_{\ell^{\prime}}\}
\]
for some symmetric matrix $\Sigma^{(\theta)}\in \R^{S^d\times S^d}$. We define $\F_{2}(r,\{\nu_m\}):=\{\CC_{\theta}:\theta\in \{0,1\}^{\gamma_N} \}$.
Lemmas~\ref{lemma:banding lower bound inclusion 2} and \ref{lemma:lower_bound2} then establish, using Proposition~\ref{prop:reduction},
 that $\mathcal{F}_2(r, \{\nu_m\}) \subseteq \mathcal{F}(r, \{\nu_m\})$ and
\begin{align*}
\inf_{\widehat{\mathcal{C}}} \sup_{\mathcal{C} \in \mathcal{F}_2(r, \{\nu_m\})} \frac{\mathbb{E} \| \widehat{\mathcal{C}} - \mathcal{C} \|}{\|\mathcal{C} \|} &\gtrsim \varepsilon_*.
\end{align*}

\paragraph{Case 3 ($r < m_*^d$).}
Let $S^d = r$, $\gamma = (S/2)^d = r / 2^d$, and define, for each $\theta \in \{0,1\}^\gamma$ the kernel integral operator $\CC_\theta$ with kernel function 
\[
k_\theta(x,y) := \sum_{\ell=1}^{S^d} \mathbf{1}\{x, y \in I_\ell\} + \tau \frac{1}{\sqrt{N r}} \sum_{\ell=1}^\gamma \theta_\ell \left( \mathbf{1}\{x \in I_\ell\} \mathbf{1}\{y \in T_x\} + \mathbf{1}\{y \in I_\ell\} \mathbf{1}\{x \in T_y\} \right),
\]
for $\tau>0$ a small constant. The domains $\{I_{\ell}\}_{\ell=1}^{S^d}$ are defined as in Case 2, and the domains $\{I_{\ell}\}_{\ell=1}^{\gamma}$ are given by 
\[
\{I_{\ell}\}_{\ell=1}^{\gamma}=\left\{\bigotimes_{\ell=1}^{d} \left[\frac{i_{\ell}}{S},\frac{i_{\ell}+1}{S}\right]: i_{\ell}\in \Big\{0,1,\ldots,\frac{S}{2}-1\Big\}\right\}.
\]
For $x \in \left[0,\frac{1}{2}\right]^d$, set $T_{x}:=\bigotimes_{i=1}^d \left[\frac{1+\lceil S x_i \rceil}{S},1\right]$. We define $\F_{3}(r,\{\nu_m\}):=\{\CC_{\theta}:\theta\in \{0,1\}^{\gamma} \}$. An argument analogous to that for $\mathcal{F}_2$ implies that $\mathcal{F}_3(r, \{\nu_m\}) \subseteq \mathcal{F}(r, \{\nu_m\})$ and Lemma~\ref{lemma:lower_bound3} shows that 
\begin{align*}
\inf_{\widehat{\mathcal{C}}} \sup_{\mathcal{C} \in \mathcal{F}_3(r, \{\nu_m\})} \frac{\mathbb{E} \| \widehat{\mathcal{C}} - \mathcal{C} \|}{\|\mathcal{C} \|} &\gtrsim \sqrt{\frac{r}{N}}.
\end{align*}
\end{proof}

\subsection{Lengthscale, effective dimension, and kernel decay}\label{subsection:lengthscale-banded}
Our upper bound in Theorem \ref{thm:bandable_upper_bound} utilizes the tapering parameter $\kappa = m_* r(\CC)^{-1/d},$ while our matching lower bound in Theorem \ref{thm:bandable} demonstrates the fundamental role that 
the effective dimension $r(\CC)$ and the decay sequence $\{\nu_m\}$ play in the estimation problem. In this section, we provide further intuition on the key quantities $r(\CC)$ and $\{\nu_m\}$ by relating them to the correlation lengthscale and to the tail decay of the kernel. First, we formalize the notion of correlation lengthscale by means of the following assumption, which, while restrictive, is often invoked in applications \cite{stein2012interpolation, williams2006gaussian}.

\begin{assumption}\label{assumption:isotropicandlengthscale}
    The kernel $k:D \times D \to \R$ satisfies:
    \begin{enumerate}[label=(\roman*)]
        \item $k= k_\lambda$ depends on a correlation lengthscale parameter $\lambda > 0,$ so that $k_\lambda(x,y) = \mathsf{K}(\|x-y\|/\lambda)$ for an isotropic base kernel $\mathsf{k}: \R^d \times \R^d \to \R$ with $\mathsf{k}(x,y) = \mathsf{K}(\|x-y\|).$ 
       \item The base kernel $\mathsf{k}$ is positive, so that $\mathsf{k}(x,y) = \mathsf{K}(\|x-y\|)>0$. Further, $\mathsf{K}(r)$ is differentiable, strictly decreasing on $[0,\infty)$, and satisfies $\lim_{r \to \infty} \mathsf{K}(r) = 0$. 
    \end{enumerate}
\end{assumption}

\begin{example}
Many widely used families of kernels are parameterized by a lengthscale parameter, including the squared exponential $k^{\text{SE}}$ and Mat\'ern $k^{\text{Ma}}$ covariance models \cite{williams2006gaussian,stein2012interpolation}
\begin{align}\label{eq:SEandMaterncovariance}
k^{\text{SE}}_\lambda(x, y) &:=\exp \left(-\frac{\|x-y\|^2}{2\lambda^2}\right), \\
k^{\text{Ma}}_\lambda(x, y) &:=\frac{2^{1-\nu}}{\Gamma(\zeta)}\left(\frac{\sqrt{2\zeta}}{\lambda}\|x-y\|\right)^\zeta K_{\zeta}\left(\frac{\sqrt{2\zeta}}{\lambda}\|x-y\|\right), 
\qquad \zeta >\frac{d-1}{2}\lor \frac{1}{2}.
\end{align}
 Here, $\Gamma$ denotes the Gamma function, $K_\zeta$ denotes the modified Bessel function of the second kind, and $\zeta$ controls the smoothness of sample paths in the Matérn model. In these and other examples, the lengthscale $\lambda$ parameterizes the covariance function, and can be heuristically interpreted as the largest distance in physical space at which correlations are significant. Lengthscale parameters are also used to define covariance operators directly. For instance, in the stochastic partial differential equation (SPDE) approach \cite{lindgren2011explicit}, Matérn-type Gaussian processes are defined by, see e.g. equation (2.4) in \cite{sanz2022spde},
\begin{equation}\label{eq:SPDE}
u\sim \mathrm{GP}(0,\CC),\quad \CC=\tau^{2s-d}(\tau^2 I-\mathcal{L})^{-s},
\end{equation}
where $\tau=\lambda^{-1}$ represents an inverse lengthscale, $\tau^{2s-d}$ is a normalizing constant to ensure $\rm{Tr}(\CC) \asymp 1$ as $\lambda \to 0,$ and $s$ is a smoothness parameter. In contrast to the setting in Assumption \ref{assumption:isotropicandlengthscale}, processes defined through the SPDE approach are typically nonstationary (and hence nonisotropic) due to boundary conditions and spatially-varying coefficients in the elliptic operator $-\mathcal{L}$. This perspective suggests via Karhunen–Lo\'eve expansion \cite{pavliotis2014stochastic} another interpretation of $\lambda$ as determining the number of eigendirections that have significant variance, and thus the effective number of frequencies that are superimposed in sample paths from the process. 
\end{example}

 The following result shows that in the small lengthscale regime where Assumption \ref{assumption:isotropicandlengthscale} holds and $\lambda$ is small, we have that  $r(\CC)^{-1/d} \asymp \lambda.$ 
 This has several important implications. First, it motivates the condition in Assumption \ref{assumption:main1} (ii), where the domain of integration is then simply determined by the correlation lengthscale. Second, it intuitively explains the choice of tapering parameter $\kappa \asymp m_* r(\CC)^{-1/d} \asymp m_* \lambda,$ whose size is determined by the correlation lengthscale.  Third, it allows us to clearly contrast the performance of the tapering and sample covariance estimators in the small lengthscale regime, as summarized in the following corollary:

\begin{corollary}\label{cor:SmallLengthscaleBanded}
Let Assumptions~\ref{assumption:main1} and \ref{assumption:isotropicandlengthscale} hold. 
      There exists a universal constant $\lambda_0>0$ such that, for any $\lambda < \lambda_0$, the
      covariance operator satisfies $r(\CC) \asymp \lambda^{-d}.$ Consequently, the sample covariance estimator and tapering estimator with $\kappa = m_* \lambda$ satisfy
     \begin{align}
         \frac{\E\|\hat{\CC}-\CC\|}{\|\CC\|} 
         &\asymp 
          \sqrt{\frac{\lambda^{-d}}{N}}
         \lor 
         \frac{\lambda^{-d}}{N} 
         , \label{eq:SampleCovBoundSmallLengthscale}\\
         \frac{\E\|\hat{\CC}_{\kappa}-\CC\|}{\|\CC\|} 
         &\lesssim
         \varepsilon^*
         + \biggl( 
         \sqrt{\frac{\log (\lambda^{-d}) }{N}}
         \lor
         \frac{\log (\lambda^{-d})}{N} \biggr)
         \label{eq:BandedBoundSmallLengthscale}
         .
     \end{align}
     In particular, if $\nu_m = m^{-\alpha}$ then $\varepsilon_* \asymp N^{-\frac{\alpha}{2\alpha+d}},$ and if $\nu_m \asymp e^{-m^t}$ then $\varepsilon_* \asymp \sqrt{(\log N)^{d/t} / N}.$ 
\end{corollary}
\begin{proof}
    By \cite[Theorem 2.8]{al2025covariancea}, for any covariance operator $\CC$ with covariance function $k$ satisfying Assumption~\ref{assumption:isotropicandlengthscale} it holds that, for small $\lambda,$ $\mathrm{Tr}(\CC) = 1$, $\|\CC\| \asymp \lambda^d,$ and $r(\CC) \asymp \lambda^{-d}.$ The bound \eqref{eq:SampleCovBoundSmallLengthscale} follows directly by \cite[Theorem 2.8]{al2025covariancea}. The bound \eqref{eq:BandedBoundSmallLengthscale} follows by plugging the values of $r(\CC)$ into the bound derived in Theorem~\ref{thm:bandable_upper_bound}.
\end{proof}

\begin{remark}
While we choose to work under Assumption \ref{assumption:isotropicandlengthscale} for exposition purposes, the scaling $r(\CC)^{-1/d} \asymp \lambda$ can also be immediately verified for SPDE covariance models as formally introduced in \eqref{eq:SPDE}. Indeed, as $\lambda\to 0$, one can directly verify via a Karhunen–Lo\'eve expansion that $\mathrm{Tr}(\CC) \asymp 1$ and that
$\|\CC\| \asymp \lambda^{-(2s-d)} (\lambda^{-2})^{-s}\asymp \lambda^d.$ 
\end{remark}

We conclude this section with a lemma which demonstrates that, in the small lengthscale regime, the sequence $\{\nu_m\}$ in the definition of our banded covariance class is determined by the tail behavior of the covariance function.
\begin{lemma}\label{lem:characterize-nu_m}
    Under Assumption~\ref{assumption:isotropicandlengthscale}, and for all sufficiently small $\lambda$,
    Assumption~\ref{assumption:main1}~(ii) is satisfied with
    \begin{align*}
        \nu_m = c \int_m^\infty r^{d-1} \mathsf{K}(r) dr, \qquad m=1,2,\dots,
    \end{align*}
    where $c$ is a universal constant chosen to ensure that $\nu_1=1$.
\end{lemma}
\begin{proof}
As in the proof of Corollary~\ref{cor:SmallLengthscaleBanded}, for small $\lambda,$ $\mathrm{Tr}(\CC) = 1$, $\|\CC\| \asymp \lambda^d,$ and $r(\CC) \asymp \lambda^{-d}.$  Therefore, to characterize the sequence $\{\nu_m\}$ in Assumption~\ref{assumption:main1}, we note that 
\begin{align*}
    &\sup_{x\in D} \int_{\{y: \|x-y\| > m r(\CC)^{-1/d}\}} k_\lambda(x,y) \, dy
    =
    \sup_{x\in D} \int_{\{y: \|x-y\| > m \lambda \}} \mathsf{K}(\|x-y\|/\lambda) \, dy\\
    & \hspace{3.8cm} \le 
     \int_{\R^d} \mathsf{K}(\|y\|/\lambda) \indicator_{\{\|y\| > m \lambda\}}(y) \,dy
     =
     \lambda^d \int_{\R^d} \mathsf{K}(\|y\|) \indicator_{\{\|y\| > m \}}(y) \, dy,
\end{align*}
where the first inequality follows by Assumption~\ref{assumption:isotropicandlengthscale}, and the last equality by the substitution $y \mapsto \lambda^{-1}y$. Switching to polar coordinates then yields 
\begin{align*}
    \lambda^d \int_0^\infty r^{d-1} \int_{S_{d-1}} \mathsf{K}(r \|u\|) \indicator_{\{r\|u\| > m \}}(u) d s_{d-1}(u) \, dr
    &= \lambda^d A(d) \int_m^\infty r^{d-1} \mathsf{K}(r) \, dr,
\end{align*}
where $S_{d-1}$ is the unit sphere, $ds_{d-1}$ is the corresponding spherical measure, and $A(d)$ is the surface area of the unit sphere in $\R^d$. To conclude, note again that $\|\CC\| \asymp \lambda^d$ for small $\lambda$.
\end{proof}

 For example, Lemma \ref{lem:characterize-nu_m} implies that for $k=k^{\text{SE}}$ with $d=1$, we have that $k_1(r) = e^{-r^2/2},$ and straightforward computations show that $\nu_m \lesssim ce^{-m^2/2}.$ Therefore, \eqref{eq:BandedBoundSmallLengthscale} yields that
 \[
  \frac{\E\|\hat{\CC}_{\kappa}-\CC^{\text{SE}}\|}{\|\CC\|} \lesssim \sqrt{\frac{(\log N)^{d/2}}{N}}
         + \biggl( 
         \sqrt{\frac{\log (\lambda^{-d}) }{N}}
         \lor
\frac{\log (\lambda^{-d})}{N} \biggr).
 \]
  In Section~\ref{sec:Numerics}, we further discuss Corollary~\ref{cor:SmallLengthscaleBanded} and Lemma \ref{lem:characterize-nu_m} in a numerical example.

\section{Estimating sparse covariance operators}\label{sec:sparse}

\subsection{Upper bound}\label{ssec:upperboundsparse}
For a Gaussian process on $D=[0,1]^d$ taking values in $\R$ with covariance function $k$, we define
\[
\|k\|^q_q:=\sup_{x\in D} \int_{D} |k(x,y)|^q dy,
\qquad \|k\|_{\infty}:=\sup_{x,y\in D} |k(x,y)|.
\]
In this section, we  invoke an approximate sparsity assumption that will be formalized through the following notions of \emph{$L^q$-sparsity} and \emph{capacity} of a covariance operator.
\begin{definition}
The \emph{$L^q$-sparsity} for $q\in [0,1]$ and \emph{capacity} of a covariance operator $\CC$ with kernel $k$ are defined respectively as
\[
\Gamma_1(q,\CC):= \frac{\|k\|_q^q\|k\|_{\infty}^{1-q}}{\|\CC\|}, 
\qquad 
\Gamma_2(\CC):=\frac{\E_{u\sim \mathrm{GP}(0,\CC)}[\sup_{x\in D} u(x)]}{\sqrt{\|k\|_{\infty}}}.
\]
\end{definition}
Notice that both $\Gamma_1(q,\CC)$ and $\Gamma_2(\CC)$ are \emph{dimension free} and \emph{scale invariant}. The following lemma, proved in Appendix \ref{app:B}, shows that, similar to $r(\CC),$ $\Gamma_1(q,\CC)$ is bounded below by $1.$ 
\begin{lemma}\label{lemma:Gamma1}
For $q\in [0,1]$, it holds that $\Gamma_1(q,\CC)\ge  \Gamma_1(1,\CC)\ge 1$.
\end{lemma}

\begin{remark} The capacity $\Gamma_2(\CC)$ is closely related to the notion of \emph{stable dimension}. Recall that for a compact set  $T \subset \R^p$  and $g\sim \mcN(0,I_p)$, the stable dimension of $T$ \cite[Section 7]{vershynin2018high} is the squared ratio of its \emph{Gaussian width} to its \emph{radius},
\[
d(T):=\left(\frac{\E\sup_{t\in T} \,\langle g,t  \rangle}{\sup_{t\in T} \|t\|}\right)^2.
\]
The squared version of $\Gamma_2(\CC)$ admits the analogous form \cite[Proposition 3.1]{al2025covariancea}
\[
\Gamma^2_2(\CC)=\frac{\left(\E_{u\sim \mathrm{GP}(0,\CC)}[\sup_{x\in D} u(x)]\right)^2}{\|k\|_{\infty}}\asymp\left(\frac{\E \sup_{f\in \F} \langle f,u\rangle }{\sup_{f\in \F}\|f\|_{\psi_2}}\right)^2,
\]
where $\mathcal{F}:=\left\{\ell_x\right\}_{x \in D}$ denotes the family of evaluation functionals, i.e. $\langle \ell_x,u \rangle=u(x)$, and $\psi_2$ denotes the Orlicz norm with Orlicz function $\psi(x) = e^{x^2}-1,$ see e.g. \cite[Definition 2.5.6]{vershynin2018high}. Comparing $\Gamma_2^2(\CC)$ with $d(T)$, it is clear that $\Gamma_2^2(\CC)$ naturally generalizes the stable dimension as it characterizes the complexity for more general and abstract Gaussian processes as opposed to the canonical Gaussian process on $T$, i.e. $\langle g,t \rangle$. The capacity $\Gamma_2(\CC)$ is rooted in deep chaining results  \cite[Theorem 1.13]{mendelson2016upper} (see also \cite{mendelson2010empirical}). The paper \cite{koltchinskii2017concentration} used these empirical process results to obtain dimension-free bounds for the sample covariance operator \cite[Theorem 4]{koltchinskii2017concentration}. Subsequently, \cite{al2024non,al2025covariancea,al2025covarianceb} used similar techniques to obtain dimension-free bounds for various thresholding matrix and operator estimators.
\end{remark}

\begin{assumption}\label{assumption:main2}
The data $\{u_n\}_{n=1}^N$ consists of $N$ independent copies of a real-valued, centered  Gaussian process $u \sim \mathrm{GP}(0,\CC)$ that is Lebesgue-almost everywhere continuous on $D = [0,1]^d$ with probability 1. 
  Moreover: 
    \begin{enumerate}[label=(\roman*)]
    \item There exists a constant $C_0>1$ such that $\Gamma_1(0,\CC) \exp(-C_0 \Gamma_2^2(\CC))\le 1.$
    \item  
    The sample size satisfies 
    $\sqrt{N}\ge \Gamma_2(\CC)\ge \frac{1}{\sqrt{N}}.$
    \end{enumerate}
\end{assumption}

For a tunable parameter $\rho>0,$ we define the thresholding estimator as
\[
\hat{k}_{\rho}(x,y):=\hat{k}(x,y)\indicator_{\{|\hat{k}(x,y)|\geq \rho\}}(x,y),\qquad (\hat{\CC}_{\rho}\psi)(\cdot) := \int_D \hat{k}_{\rho}(\cdot, y)\psi(y) \, dy, \ \ \psi\in L^2(D). 
\]

The upper bound achieved by this estimator is similar to \cite[Theorem 2.2]{al2025covariancea}, but now written in terms of $\Gamma_1(q,\CC)$ and $\Gamma_2(\CC),$ the two quantities that determine the minimax complexity in the new lower bound in Section \ref{ssec:lowerboundthresholding}. For completeness, we include a proof in Appendix \ref{app:B}.
\begin{theorem}\label{thm:sparse_upper_bound}
Under Assumption \ref{assumption:main2}, there exists an absolute constant $c>0$ such that the following holds. Let $\frac{C_0}{c}\le c_0\le \sqrt{N}$ and set 
\begin{align}\label{eq:thresholdparameter}
    \widehat{\rho} :=  \frac{c_0\sqrt{\|k\|_{\infty}}}{\sqrt{N}}\Big(\frac{1}{N}\sum_{n=1}^{N}\sup_{x\in D} u_{n}(x)\Big). 
\end{align}
Then,
\[
\frac{\E \|\widehat{\CC}_{\widehat{\rho}}-\CC\|}{\|\CC\|}\lesssim \Gamma_{1}(q,\CC)\left(\frac{\Gamma_2(\CC)}{\sqrt{N} }\right)^{1-q}.
\]
\end{theorem}

\subsection{Lower bound}\label{ssec:lowerboundthresholding}
Here, we prove a matching lower bound for covariance operator estimation over the approximate sparsity class
\begin{align}
    \F\left(\Gamma_1(q),\Gamma_2 \right):=\left\{\CC : \Gamma_1(q,\CC)\leq \Gamma_1(q), \ \Gamma_2(\CC)\leq \Gamma_2\right\}.
\end{align}
\begin{theorem}\label{thm:approximate sparsity lower bound}
    Let $1\leq N^{\beta}\leq \lfloor\exp (\frac{1}{2}\Gamma_2^2)\rfloor-1$ for some $\beta>1$, $\Gamma_2\geq 2$, and $\Gamma_1(q)\leq MN^{(1-q)/2}\Gamma_2^{-3+q}$. The minimax risk for estimating the covariance operator over $\F(\Gamma_1(q),\Gamma_2)$ in the operator norm satisfies
    \begin{align*}
        \inf_{\hat{\CC}}\sup_{\CC\in \F(\Gamma_1(q),\Gamma_2)}\frac{\mathbb{E}\|\hat{\CC}-\CC\|}{\|\CC\|}\gtrsim \Gamma_1(q)\left( \frac{\Gamma_2}{\sqrt{N}}\right)^{1-q}.
    \end{align*}
 
    \begin{proof}[Proof of Theorem \ref{thm:approximate sparsity lower bound}]
    We mirror the high-dimensional analysis in \cite{cai2012optimal} to construct a finite subclass of covariance operators $\F_1(r,\eps_{N,r})\subseteq \F(\Gamma_1 (q),\Gamma_2)$ and derive a lower bound over this class with the reduction technique from Proposition \ref{prop:reduction}. Take $r= \bigl\lfloor \exp (\frac{1}{2}\Gamma_2^2) \bigr\rfloor-1$, and now let $I_1,\ldots ,I_{r+1}$ be a uniform partition of $D=[0,1]^d$ with $\text{Vol}(I_i)=\frac{1}{r+1}$ and $\text{diam} (I_i)\asymp (r+1)^{-1/d}.$ Let $r^*=\lfloor r/2\rfloor$ and define $\Lambda$ to be the set of all $r^* \times r$ matrices whose columns are constrained to have at most $2\ell$ nonzero entries and whose rows are given by $r$-dimensional vectors of the form $[0_{r-r^*}^\top, b^\top]^\top$ where $b \in \{0,1\}^{r^*}$ satisfies $1^\top_{r^*} b = \ell$ for an integer $\ell$ to be chosen later. For any $\lambda \in \Lambda$ with row vectors denoted $\{\lambda^j\}_{j=1}^{r^*}$, we construct a corresponding $r \times r$ matrix $A_j(\lambda^j)$ that is zero everywhere except for its $j$-th row and $j$-th column, both of which are set to $\lambda^j.$ Importantly, all column and row sums of $\sum_{j=1}^{r^*}A_j(\lambda^j)$ are at most $2\ell$. Letting $\Xi=\{0,1\}^{r^*}$, we consider the parameter space

\begin{align}\label{eq:approximate sparsity parameter space}
    \Theta:=\Xi \otimes \Lambda,
\end{align}
   and construct the covariance matrix class
    \begin{align}
        \F_0 :=\left\{\Sigma(\theta)=\begin{bmatrix}
            1 & 0_{r}^\top\\
            0_{r} & \quad\frac{1}{2} \bigl(I_{r}+\eps_{N,r}\sum_{m=1}\xi_mA_m(\lambda^m)\bigr)
        \end{bmatrix}, ~\theta=(\xi,\lambda)\in \Theta\right\}
        ,
    \end{align}
     where $\eps_{N,r}$ is a constant to be chosen later. Elements of $\F_0$ are $r+1\times r+1$ matrices with bottom right $r\times r$ sub-block being identical (up to a constant scaling of $1/2$) to the choice used to prove the lower bound in \cite[Theorem 2]{cai2012optimal}.  
     Each matrix in this sub-block has diagonal elements equal to $1/2$, and contains an $r^*\times r^*$ submatrix that is potentially nonzero at the upper-right and lower-left corners, and is 0 everywhere else. Further, each row of the $r^*\times r^*$ submatrix either contains only zeros (which is the case if the corresponding entry of $\xi$ is $0$) or has precisely $\ell$ entries with value $\eps_{N,r}/2.$

      The corresponding covariance operator class is then defined to be 
    \begin{align*}
    \F_1\left(r,\eps_{N,r} \right):=\left\{\CC : k(x,y)=\sum_{i,j=1}^r\Sigma_{ij}\mathbf{1}\{x\in I_i\}\mathbf{1}\{y\in I_j\}, \ \Sigma\in \F_0\right\}.
\end{align*}
We take $\eps_{N,r}=\nu \sqrt{\frac{\log r}{N}}$, where $\nu$ is a constant satisfying $0<\nu < \inparen{\frac{1}{3M}}^{1/(1-q)}$ and $\nu^2<\frac{\beta-1}{54\beta}$. We take $\ell=\max \bigl\{ \bigl\lceil \frac{1}{2}\Gamma_1(q)\eps_{N,r}^{-q}\bigr\rceil -1,0 \bigr\}$. Lemma \ref{lemma:app sparsity lower bound inclusion} verifies the inclusion $\F_1(r,\eps_{N,r})\subseteq \F(\Gamma_1 (q),\Gamma_2)$ and Lemma \ref{lemma:approximate sparsity lower bound 2} uses Proposition \ref{prop:reduction} in conjunction with the matrix estimation lower bounds in \cite{cai2012optimal} to prove that
\begin{align*}
        \inf_{\hat{\CC}}\sup_{\CC\in \F_1(r,\eps_{N,r})}\frac{\mathbb{E}\|\hat{\CC}-\CC\|}{\|\CC\|} \gtrsim \Gamma_1(q)\left(\frac{\Gamma_2}{\sqrt{N}} \right)^{1-q}.
    \end{align*}
    \end{proof}

\end{theorem}

\subsection{Lengthscale, sparsity, and capacity}\label{subsection:lengthscale-approx}
As in the banded setting, the performance of the thresholding estimator can be studied in the small lengthscale regime by characterizing the scaling of the $L^q$-sparsity and the capacity. This analysis was developed in \cite[Section 2.2]{al2025covariancea}, and here we summarize the main result to provide a comparison with the banded case considered in  Section~\ref{subsection:lengthscale-banded}. Specifically, the following corollary is an analog of Corollary~\ref{cor:SmallLengthscaleBanded}. Similarly to the banded setting, the thresholding estimator has sample complexity that improves exponentially on that of the sample covariance. Notice that here the thresholding parameter is chosen through an empirical approximation to the expected supremum, providing an estimator that naturally adapts to the lengthscale.

\begin{corollary}[{\cite[Theorem 2.8]{al2025covariancea}}] \label{corr:LengthscaleApproxSparse}
    Suppose that Assumptions \ref{assumption:isotropicandlengthscale} and \ref{assumption:main2} hold. There is a universal constant $\lambda_0>0$ such that, for $\lambda< \lambda_0$ and $N\gtrsim \log (\lambda^{-d})$, the covariance operator satisfies for any $q \in (0,1]$ that $r(\CC) \asymp \lambda^{-d},$ $\Gamma_1(q,\mcC) \asymp 1,$ and $\Gamma_2(q,\mcC) \asymp \log(\lambda^{-d})$. Consequently, the sample covariance estimator and thresholding estimator with 
    \[
    \widehat{\rho}
    :=
    \frac{c_0}{\sqrt{N}}\Big(\frac{1}{N}\sum_{n=1}^{N}\sup_{x\in D} u_{n}(x)\Big)
    \] 
    satisfy
\begin{align*}
    \frac{\mathbb{E} \| \widehat{\CC} - \CC \|}{\|\CC\|} & \asymp  \sqrt{\frac{\lambda^{-d}}{N}} \lor \frac{\lambda^{-d}}{N},\\
    \frac{\mathbb{E} \| \widehat{\CC}_{\widehat{\rho}} - \CC \|}{\|\CC\|} &\le \,c(q)\biggl(\frac{\log(\lambda^{-d})}{N}\biggr)^{\frac{1-q}{2}},
\end{align*}
where $c_0\gtrsim 1$ is an absolute constant and $c(q)$ is a constant that depends only on $q$. 
\end{corollary}
\begin{proof}
    The result follows directly by characterizing the bound in Theorem~\ref{thm:sparse_upper_bound} in terms of the lengthscale parameter (see the proof of \cite[Theorem 2.8]{al2025covariancea}). In addition to the characterizations of $\mathrm{Tr}(\CC)$ and $\|\CC\|$ in Corollary~\ref{cor:SmallLengthscaleBanded}, it is now also necessary to characterize sharply the expected supremum, $\E[\sup_{x\in D}u(x)]$. We note that under Assumption~\ref{assumption:isotropicandlengthscale}, $\|k\|_\infty = \sup_{x\in D}k(x,x) = \mathsf{K}(0)=1.$ 
\end{proof}

\begin{remark}\label{rem:bandedVsApprox}
    Corollary \ref{corr:LengthscaleApproxSparse} implies rate $N^{(q-1)/2}$ for $q>0.$ In contrast, Corollary \ref{cor:SmallLengthscaleBanded} shows that tapering estimators can achieve rate $N^{-1/2}$ up to a logarithmic factor, provided that the kernel has fast tail decay (e.g. for squared exponential kernels). Assumption~\ref{assumption:isotropicandlengthscale} imposes a form of \emph{ordered sparsity} in that the decay of the covariance function depends monotonically on the  physical  distance between its two arguments. On the other hand, the sparsity Assumption~\ref{assumption:main2} imposes no ordering. We can therefore think of covariance operators satisfying Assumption~\ref{assumption:isotropicandlengthscale} as fitting more naturally into the class of banded operators  captured by Assumption~\ref{assumption:main1}. Corollary~\ref{corr:LengthscaleApproxSparse} demonstrates that, while more broadly useful, thresholding estimators can still perform well in the ordered setting. In Section~\ref{sec:Numerics}, we compare the performance of tapering and thresholding estimators on ordered sparse covariance operators, and show that while both do well, the tapering estimators have a clear advantage as they utilize the additional structure provided by the ordered decay. In contrast, when ordered sparsity is not present, the numerical experiments demonstrate that thresholding still does well while the tapering estimator fails. 
\end{remark}

\section{Numerical experiments}\label{sec:Numerics}
In this section, we provide a short numerical study comparing the performance of tapering and thresholding estimators. We study covariance estimation at small lengthscale for models with ordered and unordered sparse structure. For simplicity, we restrict to physical dimension $d=1$ and discretize the domain $D=[0,1]$ with a mesh of $L=1250$ uniformly spaced points. We consider a range of lengthscale ($\lambda$) parameters ranging from $10^{-3}$ to $10^{-0.1}.$ For each lengthscale, we consider the discretized version $C$ of the covariance operator of interest $\CC$, so that $C^{ij} = k(x_i,x_j)$ for $1 \le i, j \le L.$ We then generate $N= 5 \log(\lambda^{-1})$ samples of a Gaussian process on the mesh, which we denote by $u_1,\dots, u_N \iid \mcN(0,C).$ The sample covariance estimator, tapering estimator, and thresholding estimator are then defined respectively, for $1\le i,j \le L,$ by
\begin{align*}
    \hatC^{ij} &= \frac{1}{N}\sum_{n=1}^N u_n(x_i) u_n(x_j),\qquad 
    \hatC^{ij}_{\kappa} = \hatC^{ij}  f_\kappa(x_i, x_j),\qquad 
    \hatC^{ij}_{\hat{\rho}} = \hatC^{ij} \indicator_{\{|\hatC^{ij}| \ge \hat{\rho}\}},
\end{align*}
where $\kappa$ and $\hat{\rho}$ are chosen according to Corollary~\ref{cor:SmallLengthscaleBanded}
and Corollary~\ref{corr:LengthscaleApproxSparse}, respectively. The metrics of interest are the relative errors, defined for the sample, banded, and thresholded settings respectively by
\begin{align*}
     \mcE:=\frac{\|\widehat{C}-C\|}{\|C\|}, 
     \qquad 
     \mcE_{\kappa}:=\frac{\|\widehat{C}_{\kappa}-C\|}{\|C\|},
     \qquad 
     \mcE_{\widehat{\rho}}:=\frac{\|\widehat{C}_{\widehat{\rho}}-C\|}{\|C\|}.
 \end{align*}
In Figure~\ref{fig:RelErrs1}, we restrict attention to $k^{\text{SE}}$ and $k^{\text{Ma}}$ as defined in \eqref{eq:SEandMaterncovariance}. For $k^{\text{SE}}$, the banding parameter $\kappa$ is chosen according to $\nu_m = e^{-m^2/2},$ as noted in Section~\ref{subsection:lengthscale-banded} (see also Lemma~\ref{lem:characterize-nu_m}). For $k^{\text{Ma}}$, we set the smoothness parameter $\zeta = 3/2$, in which case $\nu_m = e^{-m}$. To ensure the validity of our results, each experiment is repeated a total of 30 times, and we provide averages and 95\% confidence intervals with respect to these trials. It is evident from Figure~\ref{fig:RelErrs1} that taking only $N=5 \log(\lambda^{-d})$ samples, the relative error of both the tapering and thresholding estimators significantly improve upon that of the sample covariance as the lengthscale is taken to be smaller. 

\begin{figure}
    \centering
    \includegraphics[width=0.70\linewidth]{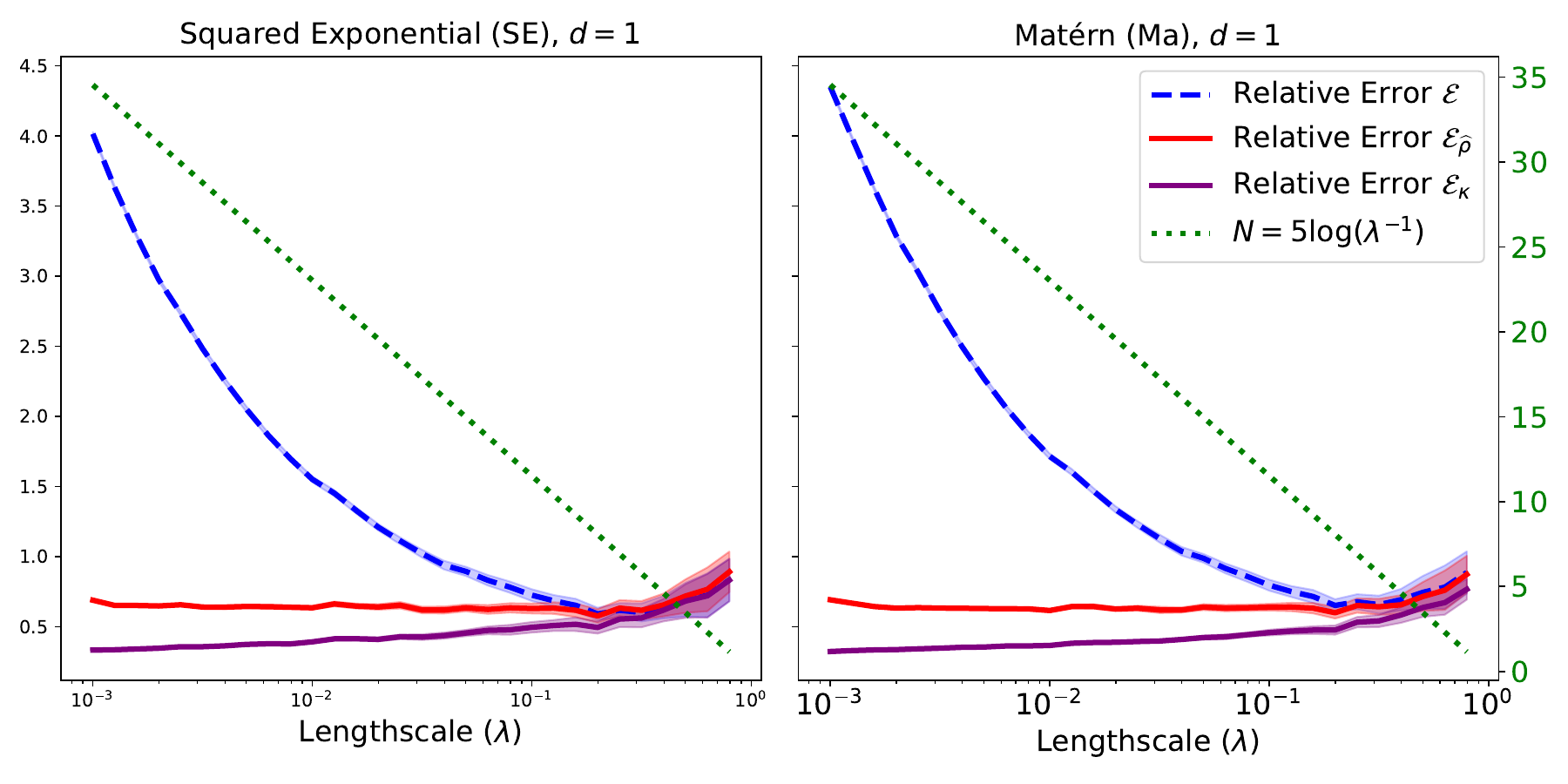}
    \caption{Plots of the average relative errors and 95\% confidence intervals achieved by the sample ($\mcE$, dashed blue), banded ($\mcE_{\kappa}$, solid purple),  thresholded ($\mcE_{\hat{\rho}}$, solid red) covariance estimators based on sample size ($N$, dotted green) for the squared exponential kernel (left) and Mat\'ern kernel (right) in physical dimension $d=1$ over 30 trials.}
    \label{fig:RelErrs1}
\end{figure}
As is to be expected (see also Remark~\ref{rem:bandedVsApprox}), although both estimators improve upon the sample covariance, the tapering estimator is superior as it exploits the underlying ordered sparsity of the covariance operators. Next, we compare the performance of the three estimators in examples with unordered structure. We first consider the periodic covariance function $k^{\text{period}}$ given by 
\begin{align*}
    k_\lambda^{\text{period}}(x,y) = \exp \inparen{-\frac{2 \sin^2(\pi \|x-y\|/\eta)}{\lambda^2}},
\end{align*}
where $\eta>0$ is the periodicity parameter. Intuitively, the periodic covariance function is composed of $\lfloor 1/\eta \rfloor$ bumps spaced uniformly over the domain, each behaving locally like $k^{\text{SE}}$. Therefore, although $k_\lambda^{\text{period}}$ is not monotonically decreasing and hence clearly violates Assumption~\ref{assumption:isotropicandlengthscale}, it is obvious that it will become more sparse as the lengthscale is taken to be smaller, and this sparsity will be unordered. In our experiments we set $\eta=0.4.$ Second, we consider the kernel $k^{\text{SE}}$ applied to a random permutation of the grid. This will preserve the sparsity but destroy the ordering. We consider now a range of lengthscale parameters $\lambda$ ranging from $10^{-2.2}$ to $10^{-0.1},$ with all other simulation parameters set to be the same as before. 
\begin{figure}
    \centering
    \includegraphics[width=0.70\linewidth]{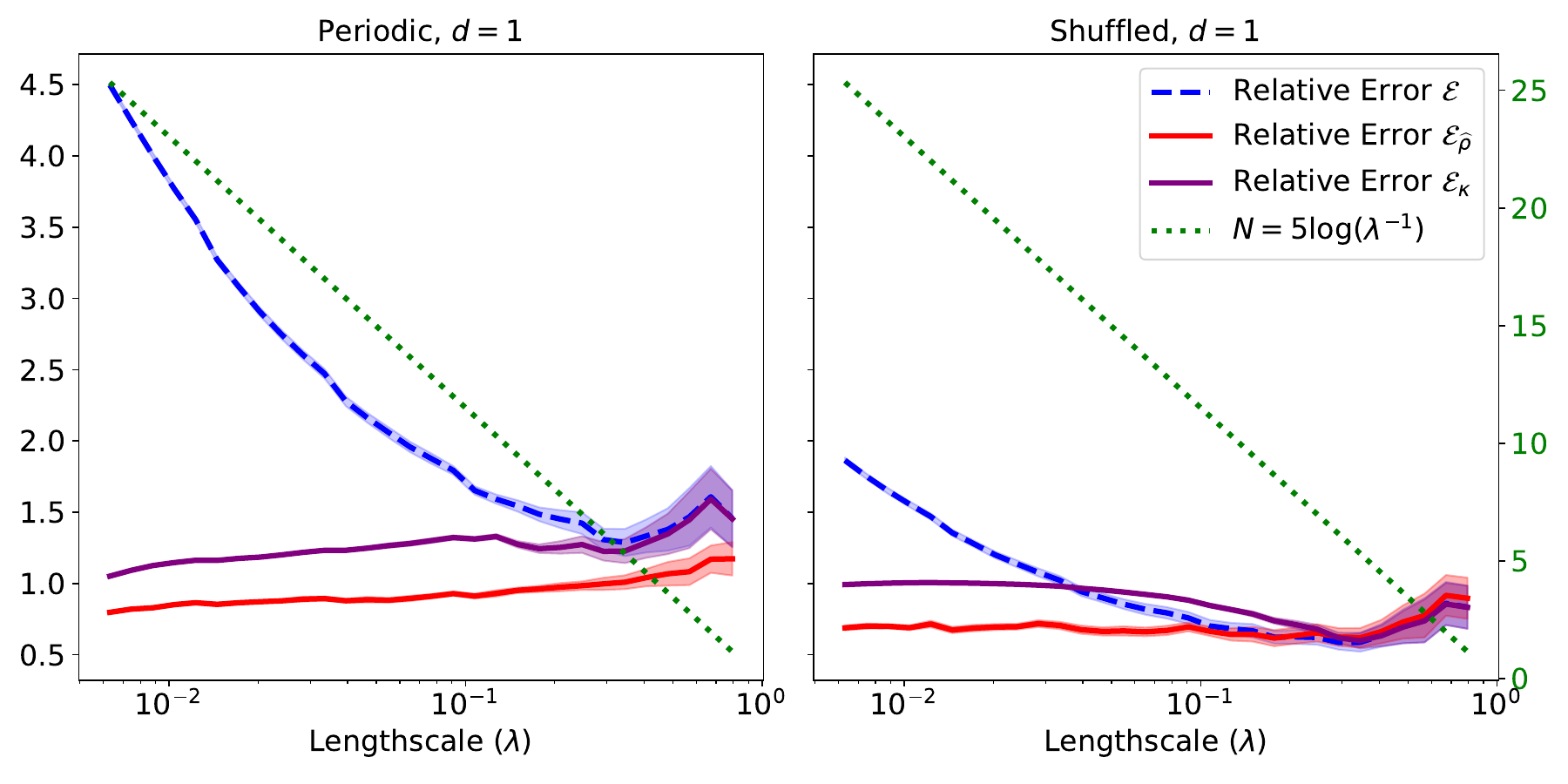}
    \caption{Plots of the average relative errors and 95\% confidence intervals achieved by the sample ($\mcE$, dashed blue), banded ($\mcE_{\kappa}$, solid purple),  thresholded ($\mcE_{\widehat{\rho}}$, solid red) covariance estimators based on sample size ($N$, dotted green) for the periodic kernel (left) and shuffled kernel (right) in physical dimension $d=1$ over 30 trials.}
    \label{fig:RelErrs2}
\end{figure}
The results are shown in Figure~\ref{fig:RelErrs2}, from which it is clear that the thresholding estimator outperforms the tapering estimator, with the latter performing worse than the zero estimator for small lengthscales. We note here that our theory is developed in the small lengthscale regime, and the behavior of the relative errors for larger lengthscales can be erratic due to the extremely small sample size.

\section{Conclusions}\label{sec:conclusions}
In this paper, we have established optimal convergence rates for estimation of banded and $L^q$-sparse covariance operators. To do so, we leveraged techniques from high-dimensional covariance matrix estimation while also addressing new challenges that emerge in the infinite-dimensional setting. Several questions stem from this work. 
\begin{enumerate}
    \item   Estimating covariance operators with non-Gaussian data (e.g., log-concave \cite{adamczak2010quantitative} and heavy-tailed distributions \cite{abdalla2022covariance}) is an interesting future direction. There is a rapidly growing body of literature on understanding the statistical and computational complexity of such tasks in the high-dimensional setting \cite{sun2016robust,wei2017estimation,ke2019user,cherapanamjeri2020algorithms}, as well as on closely related robust covariance estimation \cite{chen2018robust,mendelson2020robust,diakonikolas2023algorithmic}. 
    \item An important open direction is to investigate structured covariance operator estimation under other norms, beyond the operator norm we considered. For the sample covariance and in the unstructured setting, \cite{puchkin2023sharper} studied covariance estimation in Hilbert-Schmidt norm. In contrast to the classical matrix estimation problem, the infinite dimensional operator estimation problem lends itself naturally to study covariance estimation under norms that may account for the smoothness of the Gaussian process data. 
    \item   Another direction for future work concerns operator estimation under other structural assumptions. A natural question is estimation of Toeplitz covariance operators building on existing matrix theory \cite{cai2013optimal,klockmann2024efficient}, which may be of particular interest in time series analysis and in stationary spatial statistics. 
     Other covariance classes with broad applications include low-rank covariance in kriging for large spatial datasets \cite{cressie2008fixed,wang2022low}, and Kronecker-structured covariance models for multi-way and tensor-valued data \cite{wang2022kronecker,puchkin2024dimension}.  
    \item Finally, we conjecture that it could be possible to establish a tight connection between covariance operator estimation and existing Fourier analysis techniques for the study of kernel density estimators, see e.g. \cite[Section 1.3]{tsybakov2008introduction}. Such a connection may help intuitively explain the emergence of the nonparametric rate $N^{-\frac{\alpha}{2\alpha+d}}$ in banded matrix and operator estimation.  
\end{enumerate}

\section*{Acknowledgments}
The authors are grateful for the support of NSF DMS-2237628, DOE DE-SC0022232, and the BBVA Foundation.

\bibliographystyle{siam} 
\bibliography{references}

\begin{appendix}
 \section{Proofs of results in Section \ref{sec:banded}}\label{app:A}

\subsection{Banded covariance operators: Upper bound}\label{app:upperbandable}

The following lemma establishes an alternative representation of the tapering function $f_\kappa$ that will be useful in our analysis. 
\begin{lemma}\label{lemma:f_kappa} The tapering function $f_{\kappa}$ in \eqref{eq:tapering_func} can be written as
\begin{align*}
f_{\kappa}(x,y)=\kappa^{-d}\sum_{\sigma\in \{1,2\}^d}(-1)^{\sum_{i=1}^{d}\sigma_i}\prod_{i=1}^{d} (\sigma_i\kappa-|x_i-y_i|)_{+}.
\end{align*}
\end{lemma}

\begin{remark} When $d=1$, the representation of $f_{\kappa}$ in Lemma \ref{lemma:f_kappa} becomes
     \[
    f_{\kappa}(x,y)=\kappa^{-1}\sum_{\sigma_1\in \{1,2\}}(-1)^{\sigma_1} (\sigma_1\kappa-|x-y|)_{+}=\frac{ (2\kappa-|x-y|)_{+}-(\kappa-|x-y|)_{+}}{\kappa},
    \]
    which coincides with \cite[Lemma 1]{cai2010optimal};  see also Equation \ref{eq:taperingmatrix} in Section \ref{ssec:relatedwork}.    
\end{remark}

\begin{proof}[Proof of Lemma \ref{lemma:f_kappa}]
 Applying the equality
    \begin{align*}
    \min\left\{\frac{(2\kappa-|x_i-y_i|)_{+}}{\kappa},1\right\}&=\frac{(2\kappa-|x_i-y_i|)_{+}-(\kappa-|x_i-y_i|)_{+}}{\kappa}
    \end{align*}
and expanding the product yield that
    \begin{align*}
        f_{\kappa}(x,y)&=\prod_{i=1}^{d}\min\left\{\frac{(2\kappa-|x_i-y_i|)_{+}}{\kappa},1\right\}\\
        &=\prod_{i=1}^{d}\frac{(2\kappa-|x_i-y_i|)_{+}-(\kappa-|x_i-y_i|)_{+}}{\kappa}\\
        &=\kappa^{-d}\sum_{\sigma\in \{1,2\}^d}(-1)^{\sum_{i=1}^{d}\sigma_i}\prod_{i=1}^{d} (\sigma_i\kappa-|x_i-y_i|)_{+}.  \tag*{\qedhere}
    \end{align*}
\end{proof}


The next lemma relates the quantities $m_*$	and $\varepsilon_*$ from Definition~\ref{def:1}.

\begin{lemma}\label{lemma:basic_relation}
$m_*$ and $\varepsilon_*$ satisfy:
\begin{enumerate}[label=(\Alph*)]
\item  $N\ge 2^{-d} m_{*}^d$;
\item $\varepsilon_*= \nu_{m_*} \lor \sqrt{\frac{(m_*-1)^d}{N}}$;
\item $\varepsilon_* \le \sqrt{\frac{m_*^d}{N}}\le 2^{d/2}\varepsilon_*$.
\end{enumerate}
\end{lemma}

\begin{proof}[Proof of Lemma \ref{lemma:basic_relation}] 
  The proof follows a similar argument as \cite[Lemma C.2, Lemma 3.1]{kotekal2023minimax}. We first prove (A). If $N=1$, then $m_*=1$ and $\nu_{m_*}= \nu_1=1$, the inequality holds true. If $N\ge 2$, then $m_*\ge 2$. By definition of $m_*$, we have $\nu_{m_*-1}> \sqrt{\frac{(m_*-1)^d}{N}}$, and so \[
    1\ge \nu_{m_*-1}> \sqrt{\frac{(m_*-1)^d}{N}}\ge \sqrt{\frac{(m_*/2)^d}{N}},
    \]
    where the first inequality follows from the monotonicity of $\{\nu_m\}$ and $\nu_1=1$, and the third inequality follows by $m_*\ge 2$. Rearranging yields that $N\ge 2^{-d} m_*^d$.

    We next show (B). Consider that
    \begin{align*}
    \varepsilon_* &=\max_{m\in \N} \bigg\{ \nu_m\land \sqrt{\frac{m^d}{N}}\bigg\}=\max_{m<m_*} \bigg\{ \nu_m\land \sqrt{\frac{m^d}{N}}\bigg\}\lor \max_{m\ge m_*} \bigg\{ \nu_m\land \sqrt{\frac{m^d}{N}}\bigg\}\\
    &\overset{\text{(i)}}{=}\max_{m<m_*} \bigg\{\sqrt{\frac{m^d}{N}}\bigg\}\lor \max_{m\ge m_*} \bigg\{ \nu_m\land \sqrt{\frac{m^d}{N}}\bigg\}
    \overset{\text{(ii)}}{=}\sqrt{\frac{(m_*-1)^d}{N}}\lor \nu_{m_*},
    \end{align*}
    where in (i) we used that $\nu_m>\sqrt{\frac{m^d}{N}}$ for all $m<m_*$, and the second term in (ii) follows from the ordering of the $\{\nu_m\}$ and $\nu_{m_*}\le \sqrt{\frac{m_*^d}{N}}$.
    
    Now we prove (C). If $N=1$, then $m_*=1, \nu_{m_*}= \nu_1=1$, and so $\varepsilon_*=1$, the inequality holds true. If $N\ge 2$, then $m_*\ge 2$ and
    \begin{align*}
        \sqrt{\frac{m_*^d}{N}}\le \sqrt{\frac{2^d(m_*-1)^d}{N}}\le 2^{d/2} \varepsilon_*,
    \end{align*}
    where the second inequality follows from (B). It remains to show $\varepsilon_*\le \sqrt{\frac{m_*^d}{N}}$. For any $m\ge m_*$, we have by the monotonicity of $\{\nu_m\}$,
    \begin{align*}
        \nu_m\land\sqrt{\frac{m^d}{N}} \le \nu_m\le \nu_{m_*}\le \sqrt{\frac{m_*^d}{N}}.
    \end{align*}
    For any $m<m_*$, by definition of $\nu_{m_*}$ it follows that
    \[
     \nu_m\land\sqrt{\frac{m^d}{N}}=\sqrt{\frac{m^d}{N}}<\sqrt{\frac{m_*^d}{N}}.
    \]
    Since for all $m$ we have shown that
    \[
      \nu_m\land\sqrt{\frac{m^d}{N}}\le \sqrt{\frac{m_*^d}{N}},
    \]
    taking the maximum over $m\in \N$ yields $\varepsilon_*\le \sqrt{\frac{m_*^d}{N}}$, which completes the proof.
\end{proof}

\begin{proof}[Proof of Theorem \ref{thm:bandable_upper_bound}]
If $m_*^d >r(\CC)$, then $\kappa>1$ and the tapering estimator is equivalent to the sample covariance. Therefore, if $m_*^d >r(\CC)$ it follows from \eqref{eq:Koltchinksiibound} and Lemma \ref{lemma:basic_relation} (C) that 
\[
\frac{\mathbb{E} \| \widehat{\mathcal{C}}_{\kappa} - \mathcal{C} \|}{\|\mathcal{C} \| }  \asymp  \sqrt{\frac{r(\mathcal{C})}{N} }   \lor  \frac{r(\mathcal{C})}{N} \lesssim \sqrt{\frac{m_{*}^d}{N}} \asymp \varepsilon_*. 
\]
We henceforth assume $m_*^d\le r(\CC)$. Let $\kappa=m r(\CC)^{-1/d}$ for some $m\in \bigl[1,r(\CC)^{1/d}\bigr]$ to be determined, and consider the bias-variance trade-off:
\[
\E \|\hat{\CC}_{\kappa}-\CC\|\le 
\|\E [\hat{\CC}_{\kappa}]-\CC\| + \E \|\hat{\CC}_{\kappa}-\E [\hat{\CC}_{\kappa}]\|.
\]
We now bound the two terms in turn.

\textbf{Bias:}
By definition of $f_{\kappa}$ in \eqref{eq:tapering_func}, we have $f_{\kappa}(x,y)=1$ if $\|x-y\|_{\infty}\le \kappa
$; $f_{\kappa}(x,y)\in [0,1]$ if $\kappa<\|x-y\|_{\infty}\le 2\kappa$; and $f_{\kappa}(x,y)=0$ if $\|x-y\|_{\infty}>2\kappa$. Thus, the bias is bounded by
\begin{align}\label{eq:bias}
\begin{split}
\|\E [\hat{\CC}_{\kappa}]-\CC\| 
&\overset{\text{(i)}}{\le} \sup_{x\in D} \int_{D} |1-f_{\kappa}(x,y)| |k(x,y)|dy   \\
&= \sup_{x\in D} \bigg(\int_{\|x-y\|_{\infty}<\kappa}+ \int_{\|x-y\|_{\infty}\ge \kappa}\bigg) |1-f_{\kappa}(x,y)| |k(x,y)|dy \\
& = \sup_{x\in D} \int_{\|x-y\|_{\infty}\ge \kappa} |1-f_{\kappa}(x,y)||k(x,y)|dy  \\
&\le \sup_{x\in D} \int_{\|x-y\|_{\infty}\ge \kappa} |k(x,y)|dy 
\overset{\text{(ii)}} \le \sup_{x\in D} \int_{\|x-y\|\ge \kappa} |k(x,y)|dy  
\overset{\text{(iii)}}{\lesssim} \|\CC\| \nu_m,
\end{split}
\end{align}
where (i) follows by \cite[Lemma B.1]{al2025covariancea}, (ii) follows by $\|x-y\|_{\infty}\le \|x-y\|$, and (iii) follows from Assumption \ref{assumption:main1} \ref{assumption:sparsity}.

\textbf{Variance:} 
For a compact subset $B \subset \R^d$, we define the restriction of $\CC$ to $B\cap D$ as
\begin{align}\label{eq:restriction}
    k_{B}(x,y) := k(x,y)\mathbf{1}\{x,y\in B\cap D\}, 
    \quad
    (\CC_{B}\psi)(\cdot) := \int_D k_{B}(\cdot, y)\psi(y) \, dy,
\end{align}
for $\psi\in L^2(D).$ By definition, $\CC_{B}:L^2(D)\to L^2(D)$ is also a covariance operator.  Using the representation formula of the tapering function $f_{\kappa}(x,y)$ in Lemma \ref{lemma:f_kappa} gives that
\begin{align*}
\hat{k}_{\kappa}(x,y)&=\hat{k}(x,y)f_{\kappa}(x,y)\\
&=\sum_{\sigma\in \{1,2\}^d} \hat{k}(x,y) \kappa^{-d}(-1)^{\sum_{i=1}^{d}\sigma_i}\prod_{i=1}^{d} (\sigma_i\kappa-|x_i-y_i|)_{+}\\
&=\sum_{\sigma\in \{1,2\}^d} (-1)^{\sum_{i=1}^{d}\sigma_i} \,\,\hat{k}_{\kappa}^{(\sigma)}(x,y),
\end{align*}
where $\hat{k}_{\kappa}^{(\sigma)}(x,y):=\kappa^{-d}\,\hat{k}(x,y)\prod_{i=1}^{d} (\sigma_i\kappa-|x_i-y_i|)_{+}$.
Letting
\[
(\hat{\CC}_{\kappa}^{(\sigma)}\psi)(\cdot) := \int_D \hat{k}_{\kappa}^{(\sigma)}(\cdot, y)\psi(y) \, dy,\quad \psi\in L^2(D),
\]
we have by the triangle inequality 
\begin{align}\label{eq:decomposition}
\E \|\hat{\CC}_{\kappa}-\E [\hat{\CC}_{\kappa}]\|\le \sum_{\sigma\in \{1,2\}^d }\E \|\hat{\CC}_{\kappa}^{(\sigma)}-\E [\hat{\CC}_{\kappa}^{(\sigma)}]\|.
\end{align}
For any $x\in [0,1]$, we denote the interval $[x-\kappa,x+\kappa]$ by $B(x,\kappa)$. Using that
\begin{align*}
(\sigma_i\kappa-|x_i-y_i|)_{+} &=\mathrm{Vol}\Bigl(B(x_i,\sigma_i\kappa/2)\cap B(y_i,\sigma_i\kappa/2)\Bigr)\\
&=\int_{\theta_i\in [-\kappa,1+\kappa]} \mathbf{1}\bigl\{\theta_i\in B(x_i,\sigma_i\kappa/2)\cap B(y_i,\sigma_i\kappa/2) \bigr\} d\theta_i,
\end{align*}
we have that, for any $\sigma\in \{1,2\}^d$ and $x,y\in D$,
\begin{align*}
\prod_{i=1}^{d} (\sigma_i\kappa-|x_i-y_i|)_{+}
&=\prod_{i=1}^d \int_{\theta_i\in [-\kappa,1+\kappa]} \mathbf{1}\bigl\{\theta_i\in B(x_i,\sigma_i\kappa/2)\cap B(y_i,\sigma_i\kappa/2) \bigr\} \, d\theta_i\\
&=\int_{\theta\in [-\kappa,1+\kappa]^d } \prod_{i=1}^d \mathbf{1}\bigl\{\theta_i\in B(x_i,\sigma_i\kappa/2)\cap B(y_i,\sigma_i\kappa/2) \bigr\} \,  d\theta_i\\
&=\int_{\theta\in [-\kappa,1+\kappa]^d}  \prod_{i=1}^d \mathbf{1}\bigl\{x_i,y_i\in B(\theta_i,\sigma_i\kappa/2) \bigr\} \,  d\theta_i\\
&=:\int_{\bar{D}}\mathbf{1}\bigl\{x,y\in T(\theta) \bigr\} \, d\theta,
\end{align*}
where $\bar{D}:=[-\kappa,1+\kappa]^d$ and $T(\theta):=\bigotimes_{i=1}^{d} B(\theta_i,\sigma_i\kappa/2)=\bigotimes_{i=1}^{d}[\theta_i-\frac{\sigma_i\kappa}{2},\theta_i+\frac{\sigma_i\kappa}{2}]$. In our notation, we suppress the dependence of $T(\theta)$ on $\sigma, \kappa$ for brevity. Then, for any $\psi\in L^2(D)$,
\begin{align*}
(\hat{\CC}_{\kappa}^{(\sigma)}\psi)(x) &= \int_D \hat{k}_{\kappa}^{(\sigma)}(x, y)\psi(y) \, dy\\
&=\kappa^{-d}\int_{D}\hat{k}(x,y)\prod_{i=1}^{d} (\sigma_i\kappa-|x_i-y_i|)_{+} \psi(y) \,dy\\
&=\kappa^{-d}\int_{D} \hat{k}(x,y) \left(\int_{\bar{D}} \mathbf{1}\bigl\{x, y\in T(\theta) \bigr\} d\theta\right)\psi(y) \, dy\\
&=\kappa^{-d}\int_{\bar{D}} \left(\int_{D}\hat{k}(x,y)\mathbf{1}\bigl\{x, y\in T(\theta)\bigr\} \psi(y)dy\right) \, d\theta\\
&=\kappa^{-d}\int_{\bar{D}} \left(\int_{D} 
\hat{k}(x,y)\mathbf{1}\bigl\{x, y\in T(\theta)\cap D\bigr\} \psi(y)dy\right)d\theta\\
&=\kappa^{-d}\int_{\bar{D}} (\hat{\CC}_{T(\theta)}\psi)(x) \, d\theta.
\end{align*}
Hence, 
\[
\hat{\CC}_{\kappa}^{(\sigma)}=\kappa^{-d}\int_{\bar{D}}
\hat{\CC}_{T(\theta)} \,  d\theta,
\]
where $\hat{\CC}_{T(\theta)}$ is the restriction of $\CC$ to the domain $T(\theta)\cap D$ in the sense of \eqref{eq:restriction}, with covariance function $\hat{k}(x,y)\mathbf{1}\left\{x, y\in T(\theta)\cap D\right\}$. One can then view $\hat{\CC}_{\kappa}^{(\sigma)}$ as a mixture of covariance operators of the form $\hat{\CC}_{T(\theta)}$ with continuous uniform mixture distribution over $\theta$.

Note that $\theta = (\theta_1,\dots, \theta_d)$ with $\theta_{1},\dots, \theta_{d} \iid \mathrm{Unif}([-\kappa,1+\kappa])$. For any $1 \le i \le d,$ let $\E_i$ denote expectation with respect to $\theta_i$, and $\E_{-i}$ denote expectation with respect to $\{\theta_j\}_{j \neq i}$. Then, it follows that
\begin{align*}
\E \|\hat{\CC}_{\kappa}^{(\sigma)}-\E [\hat{\CC}_{\kappa}^{(\sigma)}]\|
&=\kappa^{-d}\,\E \left[\Big\|\int_{\bar{D}} (\hat{\CC}_{T(\theta)}- \CC_{T(\theta)})d\theta\Big\|\right]\\
&\hspace{-2cm}=\kappa^{-d}(1+2\kappa)^d\,\E \left[ \bigg\|
\E_{\{\theta_i\}\iid \mathrm{Unif}([-\kappa,1+\kappa]) }
\Big[\hat{\CC}_{T(\theta)}-\CC_{T(\theta)}\Big]\bigg\|\right]\\
&\hspace{-2cm}=\kappa^{-d}(1+2\kappa)^d\,\E \left[ \bigg\|
\E_{-i} \E_i
\Big[\hat{\CC}_{T(\theta)}-\CC_{T(\theta)}\Big]\bigg\|\right]\\
&\hspace{-2cm}=\kappa^{-d}(1+2\kappa)^d\,\E \left[ \bigg\|
\E_{-i}
\E_{\{\theta_i^{(j)}\}_{j=1}^{S}}
\Big[ \frac{1}{S}\sum_{j=1}^{S}\left(\hat{\CC}_{T([\theta_{-i},\theta_{i}^{(j)}])}-\CC_{T([\theta_{-i},\theta_i^{(j)}])}\right)\Big]\bigg\|\right], 
\end{align*}
where the last equality follows by rewriting the expectation over $\theta_i$ as an expectation over
$S$ independent copies 
$\{\theta_i^{(j)}\}_{j=1}^{S} \iid \mathrm{Unif}([-\kappa,1+\kappa])$. Now, consider as in Lemma~\ref{lemma:Q_lowerbound} the probability measure with Lebesgue density
\[
Q^{(i)}\bigl(\theta_i^{(1)},\theta_i^{(2)},\ldots,\theta_i^{(S)}\bigr)\propto \prod_{1\le s< t\le S}\mathbf{1}\left\{|\theta_i^{(s)}-\theta_i^{(t)}|>\sigma_i\kappa\right\}\prod_{1\le s\le S}\mathbf{1} \left\{\theta_i^{(s)}\in [-\kappa,1+\kappa] \right\}.
\]
By symmetry,  under $Q^{(i)}$ each $\theta_i^{(j)}$ has the same marginal distribution, which we denote by
 $Q_0^{(i)}$. It follows directly by Lemma~\ref{lemma:Q_lowerbound} with $d=1$, $S\asymp \kappa^{-1}$ and a change of measure that 
\begin{align*}
&
\kappa^{-d}(1+2\kappa)^{d}
\E \left[ \bigg\|
\E_{-i}
\E_{\{\theta_i^{(j)}\}_{j=1}^{S}}
\Big[ \frac{1}{S}\sum_{j=1}^{S}\left(\hat{\CC}_{T([\theta_{-i},\theta_{i}^{(j)}])}-\CC_{T([\theta_{-i},\theta_i^{(j)}])}\right)\Big]\bigg\|\right]  \\
&\hspace{-.25cm}=\kappa^{-d}(1+2\kappa)^{d-1}\,
\E \left[ \bigg\|
\E_{-i}
\E_{ Q^{(i)} }
\Big[ \frac{1}{S}\sum_{j=1}^{S}\frac{1}{Q^{(i)}_0(\theta_i^{(j)})}\left(\hat{\CC}_{T([\theta_{-i},\theta_{i}^{(j)}])}-\CC_{T([\theta_{-i},\theta_i^{(j)}])}\right)\Big]\bigg\|\right] \\
 &\hspace{-.25cm}=
\kappa^{-d}\,
\E \left[ \bigg\|\E_{Q^{(1)}} \cdots
\E_{ Q^{(d)} } 
\Big[ \frac{1}{S^d}\sum_{\ell=1}^{S^d}\frac{1}{\prod_{j=1}^d Q^{(j)}_0(\tilde{\theta}_j^{(\ell)})}\left(\hat{\CC}_{T(\tilde{\theta}^{(\ell)})}-\CC_{T(\tilde{\theta}^{(\ell)})}\right)\Big]\bigg\|\right], 
\end{align*}
where $\E_{Q^{(i)} }:=\E_{\{\theta_i^{(j)}\}_{j=1}^{S}\sim Q^{(i)} }$. The final equality above follows by applying the same procedure to all coordinates $1 \le i \le d$ in turn, and renaming the $S^d$ samples by $\{\tilde{\theta}^{(\ell)}\}_{1\le \ell\le S^d}$, where each $\tilde{\theta}^{(\ell)}=\bigl(\tilde{\theta}_1^{(\ell)},\tilde{\theta}_2^{(\ell)},\ldots,\tilde{\theta}_d^{(\ell)}\bigr)\in\R^d$. 

\begin{figure}[htbp]
\centering
\includegraphics[scale=0.715]{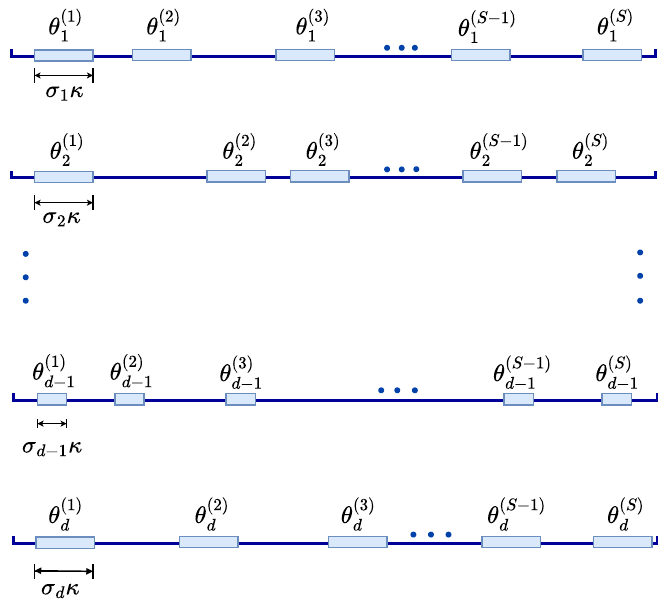}
\caption{Illustration of $\bigl\{\theta_i^{(j)}\bigr\}_{j=1}^{S}\sim Q^{(i)}$ and intervals $\bigl\{B(\theta_i^{(j)},\sigma_i\kappa/2) \bigr\}_{j=1}^{S}$, $1\le i\le d$. }
\label{fig:figure1}
\end{figure}

For each coordinate $i$, the random samples $\bigl\{\theta_i^{(j)}\bigr\}_{j=1}^{S}\sim Q^{(i)}$ satisfy $\bigl|\theta_i^{(s)}-\theta_i^{(t)}\bigr|>\sigma_i\kappa$ almost surely, namely $B\bigl(\theta_i^{(s)},\frac{\sigma_i\kappa}{2}\bigr)\cap B\bigl(\theta_i^{(t)},\frac{\sigma_i\kappa}{2}\bigr)=\emptyset$ for $s\ne t$. For $\tilde{\theta}^{(\ell)}$ and $\tilde{\theta}^{(\ell^{\prime})}$ with $\ell\ne \ell^{\prime}$, there exists at least one coordinate $\omega$ such that $\tilde{\theta}_{\omega}^{(\ell)}\ne \tilde{\theta}_{\omega}^{(\ell^{\prime})}$. This implies $B\bigl(\tilde{\theta}_{\omega}^{(\ell)},\sigma_{\omega}\kappa/2\bigr)\cap B\bigl(\tilde{\theta}_{\omega}^{(\ell^{\prime})},\sigma_{\omega}\kappa/2\bigr)=\emptyset$, hence $T\bigl(\tilde{\theta}^{(\ell)}\bigr)\cap T\bigl(\tilde{\theta}^{(\ell^{\prime})}\bigr)=\emptyset$ by the definition of $T(\theta)=\bigotimes_{i=1}^{d} B(\theta_i,\frac{\sigma_i\kappa}{2})$. Therefore, $\bigl\{T\bigl(\tilde{\theta}^{(\ell)} \bigr)\bigr\}_{1\le \ell\le S^d}$ are disjoint and by Lemma \ref{lemma:disjoint}, the last line of the above display is equivalent to
\begin{align}
&\frac{\kappa^{-d}}{S^d}\,\E \left[ 
\E_{Q^{(1)}} \cdots \E_{Q^{(d)} }\bigg[ \max_{1\le \ell\le S^d}\Big\|\frac{1}{\prod_{j=1}^d Q^{(j)}_0(\tilde{\theta}_j^{(\ell)})}\left(\hat{\CC}_{T(\tilde{\theta}^{(\ell)})}-\CC_{T(\tilde{\theta}^{(\ell)})}\right)\Big\|\bigg]\right] \nonumber\\
&\lesssim
\E \left[ 
\E_{Q^{(1)}}\cdots\E_{Q^{(d)} } 
\bigg[ \max_{1\le \ell\le S^d}\Big\|\hat{\CC}_{T(\tilde{\theta}^{(\ell)})}-\CC_{T(\tilde{\theta}^{(\ell)})}\Big\|\bigg]\right] \nonumber\\
&=
\E_{Q^{(1)}}\cdots\E_{Q^{(d)} } 
\left[\E\bigg[\max_{1\le \ell\le S^d}\Big\|\hat{\CC}_{T(\tilde{\theta}^{(\ell)})}-\CC_{T(\tilde{\theta}^{(\ell)})}\Big\|\bigg]\right] \nonumber\\
&=\E_{\{\tilde{\theta}^{(\ell)}\}_{\ell=1}^{S^d}}\left[\E\bigg[\max_{1\le \ell\le S^d}\Big\|\hat{\CC}_{T(\tilde{\theta}^{(\ell)})}-\CC_{T(\tilde{\theta}^{(\ell)})}\Big\|\bigg]\right] \nonumber,
\end{align}
where the inequality is due to the lower bound in Lemma \ref{lemma:Q_lowerbound} \ref{lemma:Q_lowerbound_b} and $S\asymp \kappa^{-1}$. We have so far shown that 
\begin{align}\label{eq:deviation1}
    \E \|\hat{\CC}_{\kappa}^{(\sigma)}-\E [\hat{\CC}_{\kappa}^{(\sigma)}]\|
    \lesssim
    \E_{\{\tilde{\theta}^{(\ell)}\}_{\ell=1}^{S^d}}\left[\E\bigg[\max_{1\le \ell\le S^d}\Big\|\hat{\CC}_{T(\tilde{\theta}^{(\ell)})}-\CC_{T(\tilde{\theta}^{(\ell)})}\Big\|\bigg]\right],
\end{align}
where the operator norm of $\hat{\CC}_{\kappa}^{(\sigma)}-\E [\hat{\CC}_{\kappa}^{(\sigma)}]$ is controlled by the maximum of operator norms of $S^d\asymp \kappa^{-d}$  covariance restrictions to disjoint small domains $T(\tilde{\theta}^{(\ell)})$ with volume roughly $\kappa^d$. This establishes a link between global estimation and local estimates. 

Next we apply the dimension-free covariance estimation result \cite{koltchinskii2017concentration} to control each $\|\hat{\CC}_{T(\tilde{\theta}^{(\ell)})}-\CC_{T(\tilde{\theta}^{(\ell)})}\|$ in the small domain, and take a union bound for the expected maximum. 

By \cite[Corollary 2]{koltchinskii2017concentration}, for all $t>1$, with probability at least $1-e^{-t}$,
\begin{align}\label{eq:upper_aux1}
\begin{split}
\|\hat{\CC}_{T(\tilde{\theta}^{(\ell)}) }-\CC_{T(\tilde{\theta}^{(\ell)})}\|
&\lesssim \|\CC_{T(\tilde{\theta}^{(\ell)})}\|\Biggl(\sqrt{\frac{r(\CC_{T(\tilde{\theta}^{(\ell)})})}{N}} \vee \frac{r(\CC_{T(\tilde{\theta}^{(\ell)})})}{N} 
\vee \sqrt{\frac{t}{N}} \vee \frac{t}{N}\Biggr).
\end{split}
\end{align}

We use the following two facts to proceed:
\begin{enumerate}[label=(\alph*)]
    \item $\|\CC_{T(\tilde{\theta}^{(\ell)})} \| \le \|\CC\|$;
    \item $\mathrm{Tr}\bigl(\CC_{T(\tilde{\theta}^{(\ell)})}\bigr)\lesssim \kappa^d\, \mathrm{Tr}(\CC)$.
\end{enumerate}  
Here (a) follows directly by the definition \eqref{eq:restriction} and (b) follows from Assumption \ref{assumption:main1} \ref{assumption:magnitude} since
\[
\mathrm{Tr}\bigl(\CC_{T(\tilde{\theta}^{(\ell)})}\bigr)=\int_{T(\tilde{\theta}^{(\ell)})\cap D} \hspace{-0.5cm} k(x,x)dx\le \Big(\sup_{x\in D}k(x,x)\Big)\mathrm{Vol}\bigl(T(\tilde{\theta}^{(\ell)})\bigr)\lesssim \kappa^d\, \mathrm{Tr}(\CC).
\]

Applying (a) and (b) to \eqref{eq:upper_aux1} gives that, for all $t>1$, with probability at least $1-e^{-t}$,
\[
\|\hat{\CC}_{T(\tilde{\theta}^{(\ell)}) }-\CC_{T(\tilde{\theta}^{(\ell)})}\| \lesssim\|\CC\|\left(\sqrt{\frac{\kappa^d r(\CC)}{N}} \vee \frac{\kappa^{d} r(\CC)}{N} \vee \sqrt{\frac{t}{N}} \vee \frac{t}{N}\right).
\]
Then, for all $t\gtrsim \|\CC\|\Big(\sqrt{\frac{\kappa^{d} r(\CC)}{N}} \vee \frac{\kappa^d r(\CC)}{N}\Big)$,
\begin{align*}
    &\P\left(\max_{1\le \ell\le S^d} \|\hat{\CC}_{T(\tilde{\theta}^{(\ell)}) }-\CC_{T(\tilde{\theta}^{(\ell)})}\|\ge t\right) \le \sum_{\ell=1}^{S^d}\P\left(\|\hat{\CC}_{T(\tilde{\theta}^{(\ell)})}-\CC_{T(\tilde{\theta}^{(\ell)})}\|\ge t\right) \\
    &\lesssim S^d \exp\left(-N\min\left\{\frac{t^2}{\|\CC\|^2},\frac{t}{\|\CC\|}\right\}\right)
    = \exp\left(-N\min\left\{t^2\|\CC\|^{-2},t\|\CC\|^{-1}\right\}+\log(S^d)\right) .
\end{align*}
Integrating the tail bound yields that 
\begin{align*}
\E \left[\max_{1\le \ell\le S^d} \|\hat{\CC}_{T(\tilde{\theta}^{(\ell)}) }-\CC_{T(\tilde{\theta}^{(\ell)})}\|\right]& \hspace{-0.1cm} = \hspace{-0.1cm}\int_{0}^{\infty}  \hspace{-0.25cm} \P\left(\max_{1\le \ell\le S^d} \|\hat{\CC}_{T(\tilde{\theta}^{(\ell)}) }-\CC_{T(\tilde{\theta}^{(\ell)})}\|>t\right)dt \\
&\lesssim \|\CC\| \left(\sqrt{\frac{\kappa^{d}r(\CC)}{N}}\lor \frac{\kappa^{d}r(\CC)}{N} + \sqrt{\frac{\log S^d}{N}} \lor \frac{\log S^d}{N} \right).
\end{align*}
The above inequality holds for every $\bigl\{\tilde{\theta}^{(\ell)}\bigr\}_{\ell=1}^{S^d}$ and $S\asymp \kappa^{-1}$. Therefore,
\begin{align}\label{eq:decomposition1}
\begin{split}
&\E_{\{\tilde{\theta}^{(\ell)}\}_{\ell=1}^{S^d}}\left[\E\bigg[\max_{1\le \ell\le S^d}\Big\|\hat{\CC}_{T(\tilde{\theta}^{(\ell)})}-\CC_{T(\tilde{\theta}^{(\ell)})}\Big\|\bigg]\right] \\
&\hspace{2.5cm}\lesssim \|\CC\| \left(\sqrt{\frac{\kappa^{d}r(\CC)}{N}}\lor \frac{\kappa^{d}r(\CC)}{N}+ \sqrt{\frac{\log (\kappa^{-d})}{N}} \lor \frac{\log (\kappa^{-d})}{N} \right).
\end{split}
\end{align}

Combining \eqref{eq:decomposition}, \eqref{eq:deviation1}, and \eqref{eq:decomposition1} gives that
\begin{align}\label{eq:variance}
\E \|\hat{\CC}_{\kappa}-\E [\hat{\CC}_{\kappa}]\|
\lesssim  \|\CC\| \left(\sqrt{\frac{\kappa^{d}r(\CC)}{N}}\lor \frac{\kappa^{d}r(\CC)}{N} +\sqrt{\frac{\log (\kappa^{-d})}{N}} \lor \frac{\log (\kappa^{-d})}{N} \right).
\end{align}

\textbf{Combining bias and variance bounds:}
Recall that $\kappa=m r(\CC)^{-1/d}$. Combining the bias bound \eqref{eq:bias} and variance bound \eqref{eq:variance} gives that
\begin{align*}
\E \|\hat{\CC}_{\kappa}-\CC\| 
\lesssim \|\CC\| \left(\nu_m+\sqrt{\frac{m^d}{N}}\lor \frac{m^d}{N} + \sqrt{\frac{\log (m^{-d}r(\CC))}{N}} \lor \frac{\log (m^{-d}r(\CC))}{N} \right).
\end{align*}
Taking $m=m_{*}\in \bigl[1,r(\CC)^{1/d}\bigr]$ and noting that $N\gtrsim m_*^d$ by Lemma \ref{lemma:basic_relation} (A), we deduce that
\begin{align*}
\E \|\hat{\CC}_{\kappa}-\CC\| &\lesssim \|\CC\|\left(\max_{m\in \N} \bigg\{ \nu_m\land \sqrt{\frac{m^d}{N}}\bigg\}+\sqrt{\frac{\log r(\CC)}{N}}\lor \frac{\log r(\CC)}{N}\right)\\
&=\|\CC\|\left(\varepsilon_*+\sqrt{\frac{\log r(\CC)}{N}}\lor \frac{\log r(\CC)}{N}\right).  \tag*{\qedhere}
\end{align*}
\end{proof}

\begin{remark}
   In combining \eqref{eq:decomposition}, \eqref{eq:deviation1}, and \eqref{eq:decomposition1},  the step \eqref{eq:variance} gives an exponential prefactor $2^d$, which is an artifact of our proof technique in which we apply the triangle inequality in \eqref{eq:decomposition} without exploiting the cancellations in the decomposition of $f_{\kappa}(x,y)$ (Lemma \ref{lemma:f_kappa}). For reference, we note that adjacent $\sigma,\sigma^{\prime}$ have different signs $(-1)^{\sum_{i=1}^d\sigma_i}$ but $\hat{k}_{\kappa}^{(\sigma)}(x,y)\approx \hat{k}_{\kappa}^{(\sigma^{\prime})}(x,y)$, so many of the terms in the summation $(-1)^{\sum_{i=1}^{d}\sigma_i} \,\,\hat{k}_{\kappa}^{(\sigma)}(x,y)+(-1)^{\sum_{i=1}^d\sigma_i^{\prime}}\,\,\hat{k}_{\kappa}^{(\sigma^{\prime})}(x,y)$ will cancel out. A more careful analysis is expected to yield a polynomial dependence on $d$. Since however we consider $d$ to be a constant throughout, we are not concerned with obtaining the sharpest dependence here. 
\end{remark}

We conclude this section with two technical lemmas that were used in the proof of Theorem \ref{thm:bandable_upper_bound}.

\begin{lemma}\label{lemma:Q_lowerbound}
   There are two constants $\kappa_0$ and $c_0$ depending only on $d$ such that the following holds. For any $0<\kappa\le \kappa_0$ and $S\le c_0\kappa^{-d}$, define \[
    Q(d\theta_1,d\theta_2,\ldots,d\theta_{S})=\frac{1}{Z}  \prod_{1\le i< j\le S}\mathbf{1}\left\{\|\theta_i-\theta_j\|>2\kappa\right\}\prod_{1\le i\le S}\mathbf{1}\{\theta_i\in \bar{D}\}d\theta_i,
    \]
    where $ Z=\int_{(\bar{D})^{\otimes S}}\prod_{1\le i< j\le S}\mathbf{1}\left\{\|\theta_i-\theta_j\|>2\kappa\right\}\prod_{1\le i\le S}d\theta_i$. Then, 
    \begin{enumerate}[label=(\alph*)]
    \item $Q$ is a probability measure over $(\bar{D})^{\otimes S}$. \label{lemma:Q_lowerbound_a}
   \item Let $Q_0$ denote the marginal probability density of $Q$. It holds that $\inf_{\theta\in \bar{D}}Q_0(\theta)\gtrsim 1.$ \label{lemma:Q_lowerbound_b}
\end{enumerate}
\end{lemma}

\begin{proof}

(a)  First, the normalization constant is finite, since
    \begin{align*}
    Z&=\int_{(\bar{D})^{\otimes S}}\prod_{1\le i< j\le S}\mathbf{1}\left\{\|\theta_i-\theta_j\|>2\kappa\right\}\prod_{1\le i\le S}d\theta_i \\
    &\le \int_{(\bar{D})^{\otimes S}}1 \prod_{1\le i\le S}d\theta_i\le (\mathrm{Vol}\bigl(\bar{D})\bigr)^S<\infty.
    \end{align*}
    Second, we prove that $Z$ has a strictly positive lower bound. For any $\{\theta_i\}_{i=1}^{S}\subseteq \bar{D}$,
    \begin{align}\label{eq:volume_bound}
    \mathrm{Vol}\left(\bar{D}\backslash \cup_{i=1}^{S} B(\theta_i,2\kappa) \right)\ge 1-S\mathrm{Vol}\bigl(B(0,2\kappa)\bigr)\ge 1-c_0\kappa^{-d}C^{\prime}\kappa^d= 1-C^{\prime}c_0, 
    \end{align}
    where $C^{\prime}$ is a constant depending only on $d$. Therefore, if $c_0<\frac{1}{C^{\prime}}$,
    \begin{align*}
        Z&=\int_{(\bar{D})^{\otimes S}}\prod_{1\le i< j\le S}\mathbf{1}\left\{\|\theta_i-\theta_j\|>2\kappa\right\}\prod_{1\le i\le S}d\theta_i\\
        &=\int_{\bar{D}} \hspace{-0.15cm} \cdots \left(\int_{\bar{D}\backslash \cup_{i=1}^{S-2} B(\theta_i,2\kappa)}\left(\int_{\bar{D}\backslash \cup_{i=1}^{S-1} B(\theta_i,2\kappa)} 1 d\theta_{S}\right)d\theta_{S-1}\right) \cdots d\theta_1\ge (1-C^{\prime}c_0)^{S}>0.
    \end{align*}
   Combining the upper and lower bound of $Z$, we have verified that $Q$ is a well-defined probability measure over $(\bar{D})^{\otimes S}$.
   
   \vspace{1em}
    
  (b)
Since $Q$ is symmetric, let $Q_0$ denote the marginal probability density of $Q$. We define the probability measure over $(\bar{D})^{\otimes (S-1)}$
\[
P(d\theta_2,d\theta_3,\ldots,d\theta_S)\propto \prod_{2\le i< j\le S}\mathbf{1}\left\{\|\theta_i-\theta_j\|>2\kappa\right\}\prod_{2\le i\le S}\mathbf{1}\{\theta_i\in \bar{D}\}d\theta_i.
\]
For any fixed $a\in \bar{D},$ we define the events $E_i=\{\|\theta_i-a\|\le 2\kappa\}$ and its complement $E_i^c=\{\|\theta_i-a\|> 2\kappa\}$ for $2\le i\le S$. We prove $\inf_{\theta\in \bar{D}}Q_0(\theta)\gtrsim 1$ by establishing the following three steps:

  \textbf{Step 1:} $Q_0(a)\gtrsim \mathbb{P}_{\{\theta_i\}_{i=2}^S\sim P}\left(\sum_{i=2}^{S}\mathbf{1}_{E_i}=0\right)$.
  
  \textbf{Step 2:} $
  \sum_{t=0}^{C}\mathbb{P}_{\{\theta_i\}_{i=2}^S\sim P}\left(\sum_{i=2}^{S}\mathbf{1}_{E_i}=t\right)=1$, where $C$ is a constant depending only on $d$.
  
 \textbf{Step 3:}
  For $1\le t\le C$, $\mathbb{P}_{\{\theta_i\}_{i=2}^S\sim P}\left(\sum_{i=2}^{S}\mathbf{1}_{E_i}=t\right)\lesssim \mathbb{P}_{\{\theta_i\}_{i=2}^S\sim P}\left(\sum_{i=2}^{S}\mathbf{1}_{E_i}=0\right).
  $

  The conclusion follows directly from these three steps:
  \[
  Q_0(a)\gtrsim\mathbb{P}_{\{\theta_i\}_{i=2}^S\sim P}\bigg(\sum_{i=2}^{S}\mathbf{1}_{E_i}=0\bigg)\gtrsim \sum_{t=0}^{C}\mathbb{P}_{\{\theta_i\}_{i=2}^S\sim P}\bigg(\sum_{i=2}^{S}\mathbf{1}_{E_i}=t\bigg)= 1.
  \]

  \textbf{Proof of Step 1:}
  By definition, 
    \begin{align*}
    Q_0(a)&=\int_{(\bar{D})^{\otimes (S-1)}} Q(a,\theta_2,\ldots,\theta_{S})\prod_{2\le i\le S}d\theta_i\\
    &=\frac{\int_{(\bar{D})^{\otimes {(S-1)}}} \prod_{2\le i\le S}\mathbf{1}\left\{\|\theta_i-a\|>2\kappa\right\} \prod_{2\le i< j\le S}\mathbf{1}\left\{\|\theta_i-\theta_j\|>2\kappa\right\}\prod_{2\le i\le S}d\theta_i}{\int_{(\bar{D})^{\otimes S}}\prod_{1\le i< j\le S}\mathbf{1}\left\{\|\theta_i-\theta_j\|>2\kappa\right\}\prod_{1\le i\le S}d\theta_i} \\
    &=\frac{\int_{(\bar{D})^{\otimes {(S-1)}}} \prod_{2\le i\le S}\mathbf{1}_{E_i^c} \prod_{2\le i< j\le S}\mathbf{1}\left\{\|\theta_i-\theta_j\|>2\kappa\right\}\prod_{2\le i\le S}d\theta_i}{\int_{(\bar{D})^{\otimes S}}\prod_{1\le i< j\le S}\mathbf{1}\left\{\|\theta_i-\theta_j\|>2\kappa\right\}\prod_{1\le i\le S}d\theta_i} \\
    &\ge \frac{\int_{(\bar{D})^{\otimes {(S-1)}}} \prod_{2\le i\le S}\mathbf{1}_{E_i^c} \prod_{2\le i< j\le S}\mathbf{1}\left\{\|\theta_i-\theta_j\|>2\kappa\right\}\prod_{2\le i\le S}d\theta_i}{\int_{(\bar{D})^{\otimes S}}\prod_{2\le i< j\le S}\mathbf{1}\left\{\|\theta_i-\theta_j\|>2\kappa\right\}\prod_{1\le i\le S}d\theta_i} \\
    &=\frac{1}{\mathrm{Vol}(\bar{D})}\frac{\int_{(\bar{D})^{\otimes {(S-1)}}} \prod_{2\le i\le S}\mathbf{1}_{E_i^c} \prod_{2\le i< j\le S}\mathbf{1}\left\{\|\theta_i-\theta_j\|>2\kappa\right\}\prod_{2\le i\le S}d\theta_i}{\int_{(\bar{D})^{\otimes (S-1)}}\prod_{2\le i< j\le S}\mathbf{1}\left\{\|\theta_i-\theta_j\|>2\kappa\right\}\prod_{2\le i\le S}d\theta_i} \\
    &=\frac{1}{\mathrm{Vol}(\bar{D})}\mathbb{P}_{\{\theta_i\}_{i=2}^S\sim P}\bigg(\prod_{i=2}^{S}\mathbf{1}_{E_i^c}=1\bigg)
    =\frac{1}{\mathrm{Vol}(\bar{D})}\mathbb{P}_{\{\theta_i\}_{i=2}^S\sim P}\bigg(\sum_{i=2}^{S}\mathbf{1}_{E_i}=0\bigg).
    \end{align*}
    
   \textbf{Proof of Step 2:}   
   For $\{\theta_i\}_{i=2}^{S}\sim P$, with probability one it holds that $\|\theta_i-\theta_j\|>2\kappa$ for $i\ne j$. Suppose $\sum_{i=2}^{S}\mathbf{1}_{E_i}=\gamma$, then there is a subset $\{\theta_\lambda\}_{\lambda\in \Lambda}\subseteq \{\theta_i\}_{i=2}^{S} $ with cardinality $|\Lambda|=\gamma$ such that $\|\theta_\lambda-t\|\le 2\kappa$ and $\|\theta_{\lambda}-\theta_{\lambda^{\prime}}\|>2\kappa$, which implies $\bigcup_{\lambda\in \Lambda}B(\theta_{\lambda},\kappa)\subseteq B(t,3\kappa)$ and $B(\theta_{\lambda},\kappa)\cap B(\theta_{\lambda^{\prime}},\kappa)=\emptyset$. A volume argument gives that $\mathrm{Vol}\bigl(B(0,3\kappa)\bigr)\ge \gamma \mathrm{Vol}\bigl(B(0,\kappa)\bigr)$, thus $\gamma\le \mathrm{Vol}\bigl(B(0,3\kappa)\bigr)/\mathrm{Vol}\bigl(B(0,\kappa)\bigr)\le C$ where $C$ is some constant depending only on $d$. Therefore, $\sum_{i=2}^{S}\mathbf{1}_{E_i}\le C$ almost surely, $\sum_{t=0}^{C}\mathbb{P}_{\{\theta_i\}_{i=2}^S\sim P}\left(\sum_{i=2}^{S}\mathbf{1}_{E_i}=t\right)=1.$

   \textbf{Proof of Step 3:}
   Given $\prod_{2\le i<j\le \omega-1}\mathbf{1}\left\{\|\theta_i-\theta_j\|>2\kappa\right\}=1$, the conditional distribution is 
\begin{align}\label{eq:conditionals}
\begin{split}
&P(\theta_\omega|\theta_2,\ldots,\theta_{\omega-1})=\frac{P(\theta_2,\theta_3,\ldots,\theta_{\omega})}{P(\theta_2,\theta_3,\ldots,\theta_{\omega-1})}  \\
&=\prod_{i=2}^{\omega-1}\mathbf{1}\left\{\|\theta_{\omega}-\theta_i\|>2\kappa\right\}\frac{\int_{\bigl(\bar{D}\backslash \cup_{i=2}^{\omega} B(\theta_i,2\kappa)\bigr)^{\otimes (S-\omega) }} \prod_{\omega+1\le i<j\le S}\mathbf{1}\{\|\theta_i-\theta_j\|>2\kappa \}\prod_{i=\omega+1}^{S} d\theta_i }{\int_{\bigl(\bar{D}\backslash \cup_{i=2}^{\omega-1} B(\theta_i,2\kappa)\bigr)^{\otimes (S-\omega+1) }} \prod_{\omega\le i<j\le S}\mathbf{1}\{\|\theta_i-\theta_j\|>2\kappa \}\prod_{i=\omega}^{S} d\theta_i}.
\end{split}
\end{align}
Suppose $S-C< \omega\le S$. The inequality \eqref{eq:volume_bound} implies that
\begin{align*}
&\int_{\bigl(\bar{D}\backslash \cup_{i=2}^{\omega} B(\theta_i,2\kappa)\bigr)^{\otimes (S-\omega) }} \prod_{\omega+1\le i<j\le S}\mathbf{1}\{\|\theta_i-\theta_j\|>2\kappa \}\prod_{i=\omega+1}^{S} d\theta_i \\
&= \int_{\bar{D}\backslash \cup_{i=2}^{\omega} B(\theta_i,2\kappa)}  \cdots \left(\int_{\bar{D}\backslash \cup_{i=2}^{S-2} B(\theta_i,2\kappa)}\left(\int_{\bar{D}\backslash \cup_{i=2}^{S-1} B(\theta_i,2\kappa)} 1\, d\theta_{S}\right)d\theta_{S-1}\right)\cdots d\theta_{\omega+1} \\
&\ge (1-C^{\prime}c_0)^{S-\omega}>(1-C^{\prime}c_0)^{C}.
\end{align*}
Moreover,
\begin{align*}
&\int_{\bigl(\bar{D}\backslash \cup_{i=2}^{\omega} B(\theta_i,2\kappa)\bigr)^{\otimes (S-\omega) }} \prod_{\omega+1\le i<j\le S}\mathbf{1}\{\|\theta_i-\theta_j\|>2\kappa \}\prod_{i=\omega+1}^{S} d\theta_i \\
&\qquad \le \int_{\bigl(\bar{D}\backslash \cup_{i=2}^{\omega} B(\theta_i,2\kappa)\bigr)^{\otimes (S-\omega) }} 1 \prod_{i=\omega+1}^{S} d\theta_i=\Bigl(\mathrm{Vol}\bigl(\bar{D}\backslash \cup_{i=2}^{\omega} B(\theta_i,2\kappa)\bigr)\Bigr)^{S-\omega}< \left(\mathrm{Vol}(\bar{D})\right)^{C}.
\end{align*}
Therefore, for $S-C<\omega\le S$, we have
\[
(1-C^{\prime}c_0)^{C}<\int_{\bigl(\bar{D}\backslash \cup_{i=2}^{\omega} B(\theta_i,2\kappa)\bigr)^{\otimes (S-\omega) }} \prod_{\omega+1\le i<j\le S}\mathbf{1}\{\|\theta_i-\theta_j\|>2\kappa \}\prod_{i=\omega+1}^{S} d\theta_i<\left(\mathrm{Vol}(\bar{D})\right)^{C}.
\]
Replacing $\omega$ by $\omega-1$, the same argument implies
\[
(1-C^{\prime}c_0)^{C+1}<\int_{\bigl(\bar{D}\backslash \cup_{i=2}^{\omega-1} B(\theta_i,2\kappa)\bigr)^{\otimes (S-\omega+1) }} \prod_{\omega\le i<j\le S}\mathbf{1}\{\|\theta_i-\theta_j\|>2\kappa \}\prod_{i=\omega}^{S} d\theta_i<\left(\mathrm{Vol}(\bar{D})\right)^{C+1}.
\]
Combining these inequalities with \eqref{eq:conditionals} yields that, for $S-C<\omega\le S$, 
\begin{align}\label{eq:conditional_bound}
\frac{(1-C^{\prime}c_0)^{C}}{\left(\mathrm{Vol}(\bar{D})\right)^{C+1}}\le P(\theta_{\omega}|\theta_2,\ldots,\theta_{\omega-1})\le \frac{\left(\mathrm{Vol}(\bar{D})\right )^C}{(1-C^{\prime}c_0)^{C+1}},\quad \text{if }\prod_{2\le i<j\le \omega}\mathbf{1}\left\{\|\theta_i-\theta_j\|>2\kappa\right\}=1.
\end{align}

 For $1\le t\le C$, we have
   \begin{align}\label{eq:auxi2}
   \begin{split}
   &\mathbb{P}_{\{\theta_i\}_{i=2}^S\sim P}\bigg(\sum_{i=2}^{S}\mathbf{1}_{E_i}=t\bigg) \\
   &\overset{(\text{i})}{=}\binom{S-1}{t}\mathbb{P}_{\{\theta_i\}_{i=2}^S\sim P}\left( \mathbf{1}_{E_i}=0 \text{ for } 2\le i\le S-t, \mathbf{1}_{E_i}=1 \text{ for } S-t+1\le i\le S\right)\\
       &= \binom{S-1}{t}\int_{(\bar{D})^{\otimes (S-1)}} \bigg(\prod_{i=2}^{S-t} \mathbf{1}_{E_i^c}\bigg) \bigg(\prod_{i=S-t+1}^{S} \mathbf{1}_{E_i}\bigg)P(\theta_2,\theta_3,\ldots,\theta_S)\prod_{2\le i\le S}d\theta_i  \\
       &=\binom{S-1}{t}\int_{(\bar{D})^{\otimes (S-t-1)}} \left(\int_{(\bar{D})^{\otimes t}}\bigg(\prod_{i=S-t+1}^{S} \mathbf{1}_{E_i}\bigg)P(d\theta_{S-t+1},\ldots,d\theta_S|\theta_2,\ldots,\theta_{S-t})\right)\\ 
       &\hspace{9cm} \bigg(\prod_{i=2}^{S-t} \mathbf{1}_{E_i^c}\bigg)  P(d\theta_2,\ldots,d\theta_{S-t})\\
       &\overset{(\text{ii})}{\lesssim}\binom{S-1}{t}\kappa^{dt}\left(\frac{\left(\mathrm{Vol}(\bar{D})\right )^{C+1}}{(1-C^{\prime}c_0)^{C+1}}\right)^{2t}   \\ & \hspace{2cm}\times \int_{(\bar{D})^{\otimes (S-t-1)}}  \left(\int_{(\bar{D})^{\otimes t}}\bigg(  \prod_{i=S-t+1}^{S}  \mathbf{1}_{E_i^c}\bigg)P(d\theta_{S-t+1},\ldots,d\theta_S|\theta_2,\ldots,\theta_{S-t}) \right)\\
       &\hspace{2cm} \times \bigg(\prod_{i=2}^{S-t} \mathbf{1}_{E_i^c}\bigg)  P(d\theta_2,\ldots,d\theta_{S-t})\\
       &=\binom{S-1}{t}\kappa^{dt}\left(\frac{\left(\mathrm{Vol}(\bar{D})\right )^{C+1}}{(1-C^{\prime}c_0)^{C+1}}\right)^{2t}\mathbb{P}_{\{\theta_i\}_{i=2}^S\sim P}\bigg(\sum_{i=2}^{S}\mathbf{1}_{E_i}=0\bigg)\\
       & \overset{(\text{iii})}{\lesssim} S^t\kappa^{dt}\left(\frac{\left(\mathrm{Vol}(\bar{D})\right )^{C+1}}{(1-C^{\prime}c_0)^{C+1}}\right)^{2t}\mathbb{P}_{\{\theta_i\}_{i=2}^S\sim P}\bigg(\sum_{i=2}^{S}\mathbf{1}_{E_i}=0\bigg) \\
   &\overset{(\text{iv})}{\lesssim} c_0^t\left(\frac{\left(\mathrm{Vol}(\bar{D})\right )^{C+1}}{(1-C^{\prime}c_0)^{C+1}}\right)^{2t}\mathbb{P}_{\{\theta_i\}_{i=2}^S\sim P}\bigg(\sum_{i=2}^{S}\mathbf{1}_{E_i}=0\bigg)
   \lesssim \mathbb{P}_{\{\theta_i\}_{i=2}^S\sim P}\bigg(\sum_{i=2}^{S}\mathbf{1}_{E_i}=0\bigg),
   \end{split}
   \end{align}
   where (i) follows by symmetry, (iii) follows from $\binom{S-1}{t}\le S^t,$ and (iv) follows from $S\le c_0\kappa^{-d}$. 
   
   Now we establish (ii). Conditioning on $\{\theta_{i}\}_{i=2}^{S-t}\sim P(\theta_2,\ldots,\theta_{S-t})$ and $\prod_{i=2}^{S-t} \mathbf{1}_{E_i^c}=1$, we shall prove
   \begin{align}\label{eq:auxi5}
   \begin{split}
  &\int_{(\bar{D})^{\otimes t}}\bigg(\prod_{i=S-t+1}^{S} \mathbf{1}_{E_i}\bigg)P(d\theta_{S-t+1},\ldots,d\theta_S|\theta_2,\ldots,\theta_{S-t}) \\
  &\lesssim \kappa^{dt}\left(\frac{\left(\mathrm{Vol}(\bar{D})\right )^{C+1}}{(1-C^{\prime}c_0)^{C+1}}\right)^{2t} \int_{(\bar{D})^{\otimes t}}\bigg(\prod_{i=S-t+1}^{S} \mathbf{1}_{E_i^c}\bigg)P(d\theta_{S-t+1},\ldots,d\theta_S|\theta_2,\ldots,\theta_{S-t}). 
  \end{split}
   \end{align}
   First,
    \begin{align}\label{eq:auxi6}
    \begin{split}
    &\int_{(\bar{D})^{\otimes t}}\bigg(\prod_{i=S-t+1}^{S} \mathbf{1}_{E_i}\bigg)P(d\theta_{S-t+1},\ldots,d\theta_S|\theta_2,\ldots,\theta_{S-t})\\
    &=\int_{\bar{D}} \cdots \left(\int_{\bar{D}} \left(\int_{\bar{D}} \mathbf{1}_{E_S}P(d\theta_S|\theta_2,\ldots,\theta_{S-1})\right) \mathbf{1}_{E_{S-1}} P(d\theta_{S-1}|\theta_2,\ldots,\theta_{S-2})\right) \\ 
    &\hspace{7.5cm} \cdots \mathbf{1}_{E_{S-t+1}} P(d\theta_{S-t+1}|\theta_2,\ldots,\theta_{S-t})\\
     &\le \int_{B(a,2\kappa)} \cdots \left(\int_{B(a,2\kappa)} \left(\int_{B(a,2\kappa)} P(d\theta_S|\theta_2,\ldots,\theta_{S-1})\right)  P(d\theta_{S-1}|\theta_2,\ldots,\theta_{S-2})\right) \\
     &\hspace{7.5cm}\cdots  P(d\theta_{S-t+1}|\theta_2,\ldots,\theta_{S-t})\\
    &\overset{\eqref{eq:conditional_bound}}{\le} \left(\frac{\left(\mathrm{Vol}(\bar{D})\right )^C}{(1-C^{\prime}c_0)^{C+1}}\right)^t \mathrm{Vol}\bigl(B(a,2\kappa)\bigr)^t 
    \lesssim \left(\frac{\left(\mathrm{Vol}(\bar{D})\right )^C}{(1-C^{\prime}c_0)^{C+1}}\right)^t \kappa^{dt}.
    \end{split}
    \end{align}
    Second,
    \begin{align}\label{eq:auxi7}
    \begin{split}
        &\int_{(\bar{D})^{\otimes t}}\bigg(\prod_{i=S-t+1}^{S} \mathbf{1}_{E_i^{c}}\bigg)P(d\theta_{S-t+1},\ldots,d\theta_S|\theta_2,\ldots,\theta_{S-t})\\
        &=\int_{\bar{D}} \cdots \left(\int_{\bar{D}} \left(\int_{\bar{D}} \mathbf{1}_{E_S^c}P(d\theta_S|\theta_2,\ldots,\theta_{S-1})\right) \mathbf{1}_{E_{S-1}^c} P(d\theta_{S-1}|\theta_2,\ldots,\theta_{S-2})\right)\\
        &\hspace{8cm}\cdots \mathbf{1}_{E_{S-t+1}^c} P(d\theta_{S-t+1}|\theta_2,\ldots,\theta_{S-t})\\
        &=\int_{\bar{D}\backslash B(a,2\kappa)} \hspace{-0.3cm} \cdots \left(\int_{\bar{D}\backslash B(a,2\kappa)} \left(\int_{\bar{D}\backslash B(a,2\kappa)}P(d\theta_S|\theta_2,\ldots,\theta_{S-1})\right)  P(d\theta_{S-1}|\theta_2,\ldots,\theta_{S-2})\right)\\
        &\hspace{8cm}\cdots P(d\theta_{S-t+1}|\theta_2,\ldots,\theta_{S-t})\\
        &\overset{\eqref{eq:conditional_bound}+\eqref{eq:volume_bound}}{\ge} \prod_{i=S-t+1}^{S} \frac{(1-C^{\prime}c_0)^{C}}{\left(\mathrm{Vol}(\bar{D})\right)^{C+1}} (1-C^{\prime}c_0) 
        =\left(\frac{(1-C^{\prime}c_0)^{C+1}}{\left(\mathrm{Vol}(\bar{D})\right)^{C+1}}\right)^{t}.
        \end{split}
    \end{align}
    Combining \eqref{eq:auxi6} and \eqref{eq:auxi7} implies \eqref{eq:auxi5}. This completes the proof of \eqref{eq:auxi2} and \textbf{Step 3}.
    \end{proof}

\begin{lemma}\label{lemma:disjoint}
    Let $\{\CC_\ell\}_{\ell=1}^{S}$ be a set of kernel integral operators on $L^2(D)$ with kernel functions $\{k_\ell(\cdot,\cdot)\}_{\ell=1}^{S}$. For $1\le \ell\le S$, we denote the support of $k_\ell(\cdot,\cdot)$ by $T_\ell\times T_\ell$. If $\{T_\ell\}_{\ell=1}^{S}\subseteq D$ are disjoint, then
    \[
    \Big\|\sum_{\ell=1}^{S}\CC_{\ell}\Big\|=\max_{1\le \ell\le S}\|\CC_{\ell}\|.
    \]
\end{lemma}

\begin{proof} 
For any $\psi\in L^2(D)$ with $\|\psi\|_{L^2(D)}=1$,
\begin{align*}
\bigg\|\bigg(\sum_{\ell=1}^{S}\CC_{\ell}\bigg)\psi\bigg\|_{L^2(D)}^2 &=\bigg\|\int_{D}\sum_{\ell=1}^{S}k_{\ell}(x,y)\psi(y)dy\bigg\|_{L^2(D)}^2
=\bigg\|\sum_{\ell=1}^{S}\int_{D}k_{\ell}(x,y)\psi(y)dy\bigg\|_{L^2(D)}^2\\
&\overset{\text{(i)}}{=}\sum_{\ell=1}^{S}\bigg\|\int_{D}k_{\ell}(x,y)\psi(y)dy\bigg\|_{L^2(D)}^2
\overset{\text{(ii)}}{=}\sum_{\ell=1}^{S}\bigg\|\int_{T_{\ell}}k_{\ell}(x,y)\psi(y)dy\bigg\|_{L^2(D)}^2\\
&\le \sum_{\ell=1}^{S} \|\CC_{\ell}\|^2 \|\psi\|_{L^2(T_{\ell})}^2
\le \bigg(\max_{1\le \ell\le S} \|\CC_{\ell}\|^2\bigg)\bigg(\sum_{\ell=1}^{S}\|\psi\|_{L^2(T_\ell)}^2\bigg)\\
&\overset{\text{(iii)}}{\le}\max_{1\le \ell\le S} \|\CC_{\ell}\|^2,
\end{align*}
where (i) follows from the fact that that $\big\{\int_{D}k_{\ell}(\cdot,y)\psi(y)dy\big\}_{\ell=1}^{S}$ have disjoint supports, (ii) follows by $k_{\ell}(x,y)=0$ if $y\notin T_{\ell}$, and (iii) follows by $\sum_{\ell=1}^{S}\|\psi\|_{L^2(T_\ell)}^2\le \|\psi\|_{L^2(D)}^2=1$. 

Therefore, $\Big\|\sum_{\ell=1}^{S}\CC_{\ell}\Big\|\le \max_{1\le \ell\le S}\|\CC_{\ell}\|$. Let $\ell^{\prime}:=\arg\max_{1 \le \ell\le S}\|\CC_{\ell}\|$, by taking $\psi\in L^2(D)$ with support $T_{\ell^{\prime}}$ and satisfying $\|\CC_{\ell^{\prime}}\psi\|=\|\CC_{\ell^{\prime}}\|\|\psi\|$, we conclude that the equality holds. 
\end{proof}

\subsection{Banded covariance operators: Lower bound}\label{app:lowerbandable}

\begin{proof}[Proof of Proposition~\ref{prop:reduction}]
    (a)
    Observe that, for any $\psi\in L^2(D)$,
\[
\langle\psi,\CC_{\Sigma}\psi\rangle_{L^2(D)}=\int_{D\times D} k_{\Sigma}(x,y) \psi(x)\psi(y) \, dxdy = \langle  \bar{\psi}, \Sigma \bar{\psi} \rangle \ge 0,
\]
where $\bar{\psi}=(\bar{\psi}_1,\ldots,\bar{\psi}_M)\in \R^M$ and $\bar{\psi}_i=\int_{I_i} \psi(x)dx$. We see that the operator $\CC_{\Sigma}$ is \textit{positive semi-definite}. Moreover, $\CC_{\Sigma}$ is \textit{trace-class}, since  
\[
\mathrm{Tr}( \CC_{\Sigma}) = \int_{D}k_{\Sigma}(x,x) \, dx=\sum_{i=1}^M \Sigma_{ii}\mathrm{Vol}(I_i)= \frac{1}{M}\sum_{i=1}^{M}\Sigma_{ii} <\infty.
\]

(b) For $\psi\in L^2(D)$ with $\|\psi\|_{L^2(D)}=1$,
\[
\|\bar{\psi}\|^2=\sum_{i=1}^{M}(\bar{\psi}_i)^2=\sum_{i=1}^{M}\bigg(\int_{I_i} \psi(x)dx\bigg)^2\le \sum_{i=1}^{M}\bigg(\int_{I_i}\psi(x)^2 dx \bigg)\mathrm{Vol}(I_i)=\frac{1}{M}.
\]
We use an equivalent characterization of the operator norm to write 
\begin{align*}
\|\CC_{\Sigma}\|
=\sup_{\|\psi\|_{L^2(D)}=1}\langle\psi,\CC_{\Sigma}\psi\rangle_{L^2(D)}
=\sup_{\|\psi\|_{L^2(D)}=1} \langle  \bar{\psi}, \Sigma \bar{\psi} \rangle\le \sup_{\|v\|\le \frac{1}{\sqrt{M}}}\langle  v, \Sigma v\rangle
=\frac{1}{M}\|\Sigma\|.
\end{align*}
On the other hand, for any unit vector $v=(v_1,v_2,\ldots, v_M)\in \R^M $, we set $\psi_v(x)=\sqrt{M}\sum_{i=1}^M v_i\mathbf{1}\{x\in I_i\}$. Observe that $\|\psi_v\|_{L^2(D)}=1$, so
\[
\|\CC_{\Sigma}\|=\sup_{\|\psi\|_{L^2(D)}=1}\langle\psi,\CC_{\Sigma}\psi\rangle_{L^2(D)}\ge \sup_{\|v\|=1} \langle \psi_v,\CC_{\Sigma}\psi_v  \rangle_{L^2(D)}=\sup_{\|v\|=1}\frac{1}{M}\langle  v, \Sigma v \rangle=\frac{1}{M}\|\Sigma\|,
\]
as desired.

(c)  We write
\begin{align}\label{eq:reduction_aux1}
  \inf_{\hat{\CC}}\sup_{\CC\in \F^{*}} \E \|\hat{\CC}-\CC\|= \inf_{\hat{k}}\sup_{\Sigma\in \F} \E \sup_{\|\psi\|_{L^2(D)}=1} \Big\|\int_D \left(\hat{k}(\cdot,y)-k_{\Sigma}(\cdot,y)\right)\psi(y)dy\Big\|_{L^2(D)}.
\end{align}
 For any unit vector $v=(v_1,v_2,\ldots, v_M)\in \R^M $, we set $\psi_v(x)=\sqrt{M}\sum_{i=1}^M v_i\mathbf{1}\{x\in I_i\}$. Observe that $\|\psi_v\|_{L^2(D)}=1$. We then restrict the supremum over all $\|\psi\|_{L^2(D)}=1$ in the lower bound \eqref{eq:reduction_aux1} to be a supremum over functions of the form $\psi_v$, which yields
\begin{align*}
\inf_{\hat{\CC}}\sup_{\CC\in \F^{*}} \E \|\hat{\CC}-\CC\|\geq \inf_{\hat{k}} \hspace{-0.1cm} \sup_{\Sigma\in\F} \E \sup_{\|v\|=1} \Big\|\int_D \hspace{-0.1cm}  \left(\hat{k}(\cdot,y)-k_{\Sigma}(\cdot,y) \right)\sqrt{M}\sum_{j=1}^M v_j\mathbf{1}\{y\in I_j\}dy\Big\|_{L^2(D)}.
\end{align*}
To simplify this expression, we observe the following:
\begin{align*}
    &\Big\|\int_D \left(\hat{k}(\cdot,y)-k_{\Sigma}(\cdot,y) \right)\sqrt{M}\sum_{j=1}^M v_j\mathbf{1}\{y\in I_j\}dy\Big\|^2_{L^2(D)} \\
    &=\int_{D} \hspace{-0.1cm} \bigg(\int_D \bigg(\hat{k}(x,y)-\sum_{i,j=1}^{M} \Sigma_{ij}\mathbf{1}\{x\in I_i\} \mathbf{1}\{y\in I_j\} \bigg)\sqrt{M}\sum_{j=1}^M v_j\mathbf{1}\{y\in I_j\}dy\bigg)^2 \hspace{-0.1cm} dx \\
    &=\sum_{i=1}^{M}\int_{I_i}\bigg(\int_D \bigg(\hat{k}(x,y)-\sum_{j=1}^{M} \Sigma_{ij} \mathbf{1}\{y\in I_j\} \bigg)\sqrt{M}\sum_{j=1}^M v_j\mathbf{1}\{y\in I_j\}dy\bigg)^2  \hspace{-0.1cm}  dx\\
    &=M\sum_{i=1}^{M}\int_{I_i} \bigg(\sum_{j=1}^{M}v_j\Big(\int_{I_j}\hat{k}(x,y)dy\Big)-\Big(\frac{1}{M}\sum_{j=1}^{M}\Sigma_{ij}v_j\Big) \bigg)^2  dx\\
    &\overset{\text{(i)}}{\ge} M\sum_{i=1}^{M} \frac{1}{\mathrm{Vol}(I_i)}\bigg(\int_{I_i}\Big(\sum_{j=1}^{M}v_j\int_{I_j}\hat{k}(x,y)dy\Big)-\Big(\frac{1}{M}\sum_{j=1}^{M}\Sigma_{ij}v_j\Big)dx \bigg)^2\\
    &= M^2\sum_{i=1}^{M}\bigg(\Big(\sum_{j=1}^{M}v_j \int_{I_i\times I_j}\hat{k}(x,y)dxdy\Big)-\Big(\frac{1}{M^2}\sum_{j=1}^{M}\Sigma_{ij}v_j\Big) \bigg)^2\\
    &=\frac{1}{M^2}\|(\hat{k}_M-\Sigma) v\|^2,
\end{align*}
 where (i) follows by Cauchy-Schwarz inequality, and we have defined  $$\hat{k}_M:=\left(M^2\int_{I_i\times I_j}\hat{k}(x,y)dxdy \right)_{1\leq i,j\leq M}$$ in the last step. Consequently, we have that 
\begin{align*}
    \inf_{\hat{\CC}}\sup_{\CC\in \F^{*}} \E \|\hat{\CC}-\CC\| \ge \inf_{\hat{k}}\sup_{\Sigma\in\F}\E \sup_{\|v\|=1}\frac{1}{M}\|(\hat{k}_{M}-\Sigma)v\|=\frac{1}{M}\inf_{\hat{k}}\sup_{\Sigma\in\F}\E \|\hat{k}_{M}-\Sigma\|,
\end{align*}
where $\inf_{\hat{k}}$ ranges over all measurable functions of $\{u_n(\cdot)\}_{n=1}^N$.

We note that the Gaussian random function $u_n(\cdot)\sim \mathrm{GP}(0,\CC_{\Sigma})$ is almost surely piecewise constant and can be written as 
\begin{align*}
u_n(x)=\sum_{i=1}^M X_{ni}\mathbf{1}\{x\in I_i\},
\end{align*}
where the $M$-dimensional vector $X_n=(X_{n1},\ldots,X_{nM})\sim \mcN(0,\Sigma)$. Therefore, $\hat{k}_M$ is simply a function of the multivariate Gaussians $X_1,X_2,\ldots,X_N\iid \mcN(0,\Sigma)$, which yields that 
\[
\inf_{\hat{\CC}}\sup_{\CC\in \F^{*}} \E \|\hat{\CC}-\CC\|\ge \frac{1}{M}\inf_{\hat{k}}\sup_{\Sigma\in\F}\E \|\hat{k}_{M}-\Sigma\|= \frac{1}{M}\inf_{\hat{\Sigma}}\sup_{\Sigma\in \F} \E \|\hat{\Sigma}-\Sigma\|.
\]
Combining the argument above with $\|\CC_{\Sigma}\|=\frac{1}{M}\|\Sigma\|$ in (b), we conclude that
\[
\inf_{\hat{\CC}}\sup_{\CC\in \F^{*}} \frac{\E\| \hat{\CC}-\CC\|}{\|\CC\|}\ge \inf_{\hat{\Sigma}}\sup_{\Sigma\in \F}\frac{\E\|\hat{\Sigma}-\Sigma\|}{\|\Sigma\|}.  \qedhere
\]
\end{proof}

\begin{lemma}\label{lemma:banding lower bound inclusion 1}
    If $N>\log r$, then $\F_1(r,\{\nu_m\})\subseteq \F(r,\{\nu_m\})$.
\end{lemma}

\begin{proof}[Proof of Lemma \ref{lemma:banding lower bound inclusion 1}]
            
\begin{enumerate}[label=(\alph*)]
\item $\Sigma_{\ell}\in \F_0$ is positive semi-definite (taking $\tau$ sufficiently small).
\item For $\CC\in \F_1(r,\{\nu_m\})$ and $\Sigma\in \F_0$, $\mathrm{Tr}(\CC)=\int_{D} k(x,x)dx=\frac{1}{r}\sum_{j=1}^{r}\Sigma_{jj}\asymp \sup_{1\le j\le r}\Sigma_{jj}= \sup_{x\in D}k(x,x)$.
\item For $\CC\in \F_1(r,\{\nu_m\})$ and $\Sigma\in \F_0$, $r(\CC)=\frac{\mathrm{Tr(\CC)}}{\|\CC\|}\overset{\text{(i)}}{=}\frac{\frac{1}{r}\sum_{j=1}^{r}\Sigma_{jj}}{\frac{1}{r}\|\Sigma\|}=\frac{w r-w\sqrt{\frac{\tau}{N}\log r}}{w}=r-\sqrt{\frac{\tau}{N}\log r}\le r$, where (i) follows by Proposition~\ref{prop:reduction} \ref{eq:reductionb}. 
\item Since $\Sigma\in \F_0$ is diagonal,
\begin{align}\label{eq:F1_aux1}
\sup_{x\in D} \int_{\{y: \|x-y\| > Mr(\CC)^{-1/d}\}} |k(x,y)| dy &=\sup_{x\in D}\int_{\{y: \|x-y\| > Mr(\CC)^{-1/d}\}} \sum_{i=1}^{r}\Sigma_{ii}\mathbf{1}\{x,y\in I_i\} dy\nonumber\\
&\le \big(\max_{1\le i\le r}\Sigma_{ii}\big)\cdot \mathrm{Vol}(I_i)=\frac{w}{r}=\|\CC\|.
\end{align}

Notice that $r(\CC)=r-\sqrt{\frac{\tau}{N}\log r}\asymp r$ and $\mathrm{diam}(I_i)\asymp r^{-1/d}$. Therefore, there exists an absolute constant $C_1>0$ such that when $m>C_1$, for any $x\in D$ (suppose $x\in I_x$), $I_x\subseteq B(x,m r(\CC)^{-1/d})$, and so
\begin{align}\label{eq:F1_aux2}
\sup_{x\in D} \int_{\{y: \|x-y\| > m r(\CC)^{-1/d}\}} |k(x,y)| dy &=\sup_{x\in D}\int_{\{y: \|x-y\| > m r(\CC)^{-1/d}\}} \sum_{i=1}^{r}\Sigma_{ii}\mathbf{1}\{x,y\in I_i\} dy\nonumber\\
&\le w \sup_{x\in D}\int_{D\backslash B(x,m r(\CC)^{-1/d}) }\mathbf{1}\{x,y\in I_{x}\}dy\overset{\text{(i)}}{=}0,
\end{align}
where (i) follows from $I_{x}\cap \left(D\backslash B(x,m r(\CC)^{-1/d}) \right)=\emptyset$. Combining \eqref{eq:F1_aux1} and \eqref{eq:F1_aux2} gives that, for $\CC\in \F_1(r,\{\nu_m\})$,
\[
\sup_{x\in D} \int_{\{y: \|x-y\| > m r(\CC)^{-1/d}\}} |k(x,y)| dy\lesssim \|\CC\|\nu_m,\quad \forall \ m\in \N. \qedhere
\]
\end{enumerate}
\end{proof}

\begin{lemma}[Lower bound over $\F_1(r,\{\nu_m\})$]\label{lemma:lower_bound1}
Suppose $N> \log r>0$. The minimax risk for estimating the covariance operator over $\F_1(r,\{\nu_m\})$ under the operator norm satisfies
\[
\inf_{\hat{\CC}}\sup_{\CC\in \F_1(r,\{\nu_m\})}\frac{\mathbb{E} \| \widehat{\mathcal{C}} - \mathcal{C} \|}{\|\mathcal{C} \| } \gtrsim \sqrt{\frac{\log r}{N}}.
\]
\end{lemma}

\begin{proof}[Proof of Lemma \ref{lemma:lower_bound1}]
We can assume without loss of generality that $r>1$ is an integer (otherwise replace $r$ with $\lceil r \rceil$). According to our construction of $\F_1(r,\{\nu_m\})$ and the inequality \eqref{eq:reduction2} in Proposition~\ref{prop:reduction}, we have the lower bound reduction
\[
\inf_{\hat{\CC}}\sup_{\CC\in \F_1(r,\{\nu_m\})}\frac{\mathbb{E} \| \widehat{\mathcal{C}} - \mathcal{C} \|}{\|\mathcal{C} \| }\ge \inf_{\hat{\Sigma}}\sup_{\Sigma\in \F_0}\frac{\mathbb{E} \| \widehat{\Sigma} - \Sigma \|}{\|\Sigma \| }.
\]
The desired lower bound follows using the same argument as in \cite[Section 3.2.2]{cai2010optimal} (Le Cam's method), noticing that the dimension of the matrix is $r$, $\|\Sigma\|\asymp w$ for every $\Sigma\in \F_0$, $\|\Sigma_0-\Sigma_\ell\|=w \sqrt{\frac{\tau}{N}\log r}$,  and the total variation distance is invariant with respect to scaling transformations, see e.g. \cite{devroye2018total}. 
\end{proof}

\begin{lemma}\label{lemma:banding lower bound inclusion 2}
    If $N>\log r$ and $r>m_*^d$, then $\F_2(r,\{\nu_m\})\subseteq \F(r,\{\nu_m\})$.
\end{lemma}

\begin{proof}[Proof of Lemma \ref{lemma:banding lower bound inclusion 2}]
    
\begin{enumerate}[label=(\alph*)]
    \item $\Sigma^{(\theta)}$ (and then $\CC_{\theta}$) is positive semi-definite: By Gershgorin's theorem, for $\tau\in (0,4^{-d-1})$,
    \begin{align*}
    \lambda_{\min}(\Sigma^{(\theta)})\ge 1-\tau h_N (2K)^d= 1-\tau 2^d \sqrt{\frac{m_*^d}{N}}\ge 1-\tau 4^d \ge \frac{3}{4}.
    \end{align*}
    \item For every $\theta\in \{0,1\}^{\gamma_N}$, $\mathrm{Tr}(\CC_{\theta})=\int_{D}k_{\theta}(x,x)dx=\int_{D}\sum_{\ell=1}^{S^d}\mathbf{1}\{x\in I_{\ell}\}dx=1=\sup_{x}k_{\theta}(x,x)$.
    \item By Proposition~\ref{prop:reduction} \ref{eq:reductionb}, $\|\CC_{\theta}\|=S^{-d}\|\Sigma^{(\theta)}\|$. Recall that $r=S^d$, we take the unit vector $e_{r}=(0,0,\ldots,0,1)\in \R^{S^d}$,
    \[
    \|\Sigma^{(\theta)}\|\ge \langle e_r, \Sigma^{(\theta)} e_r \rangle=\Sigma^{(\theta)}_{rr}=1.
    \]
    On the other hand, $\|\Sigma^{(\theta)}\|\le \|\Sigma^{(\theta)}\|_{1}\le 1+\tau h_N(2K)^d\le \frac{5}{4}$. Then, $S^{-d}\le \|\CC_{\theta}\|\le \frac{5}{4} S^{-d}$ for all $\theta\in \{0,1\}^{\gamma_N}$, and so $r(\CC_{\theta})=\frac{\mathrm{Tr}(\CC_{\theta})}{\|\CC_{\theta}\|}=\frac{1}{\|\CC_{\theta}\|}\le S^d=r$ and $r(\CC_{\theta})=\frac{1}{\|\CC_{\theta}\|}\ge \frac{4}{5}S^d =\frac{4}{5}r$.
    \item For $m>m_{*}-1=2\sqrt{d}K$, 
    \begin{align*}
&\sup_{x\in D} \int_{\{y: \|x-y\| > mr(\CC_{\theta})^{-1/d}\}} |k_{\theta}(x,y)| dy \overset{\text{(i)}}{\le} \sup_{x\in D} \int_{\{y: \|x-y\|_{\infty} > d^{-1/2} mr(\CC_{\theta})^{-1/d}\}} |k_{\theta}(x,y)| dy\\
&\overset{\text{(ii)}}{\le}  \sup_{x\in D} \int_{\{y: \|x-y\|_{\infty} > d^{-1/2}mS^{-1}\}} |k_{\theta}(x,y)| dy
\overset{\text{(iii)}}{\le}  \sup_{x\in D} \int_{\{y: \|x-y\|_{\infty} > 2KS^{-1}\}} |k_{\theta}(x,y)| dy \overset{\text{(iv)}}{=}0,
\end{align*}
where (i) follows by $\|x-y\|\le \sqrt{d}\|x-y\|_{\infty}$, (ii) follows from $r(\CC_{\theta})^{-1/d}\ge r^{-1/d}=S^{-1}$, and (iii) follows by $m>2\sqrt{d}K$; if $\|x-y\|_{\infty}>2KS^{-1}$, there exists at least one coordinate $\omega$ such that $|x_{\omega}-y_{\omega}|>2KS^{-1}$, then $k_{\theta}(x,y)=0$ following by our construction and (iv) holds.

For $\sqrt{d}\le m\le m_{*}-1$,
\begin{align*}
\sup_{x\in D} \int_{\{y: \|x-y\| > mr(\CC_{\theta})^{-1/d}\}} |k_{\theta}(x,y)| dy & \le \sup_{x\in D} \int_{\{y: \|x-y\|_{\infty} > d^{-1/2}mS^{-1}\}} |k_{\theta}(x,y)| dy\\
&\overset{\text{(i)}}{\le} \tau h_N (2K)^d S^{-d}\overset{\text{(ii)}}{\lesssim} \|\CC_{\theta}\| \nu_{m_{*}-1}\overset{\text{(iii)}}{\le} \|\CC_{\theta}\| \nu_m,
\end{align*}
where (i) follows by noticing that $\|x-y\|_{\infty}> d^{-\frac{1}{2}}mS^{-1}\ge S^{-1}$ implies $\sum_{\ell=1}^{S^d} \mathbf{1}\{x,y\in I_{\ell}\}=0$, (ii) follows by $\|\CC_{\theta}\|\asymp S^{-d}$ and
\[
h_N=K^{-d}\sqrt{\frac{m_*^d}{N}}\le K^{-d}\sqrt{\frac{2^d(m_*-1)^d}{N}} < K^{-d}2^{d/2}\nu_{m_*-1},
\]
and (iii) follows by the monotonicity of $\{\nu_m\}$ and $m\le m_{*}-1$.

For $1<m\le \sqrt{d}$,
\begin{align*}
\sup_{x\in D} \int_{\{y: \|x-y\| > m r(\CC_{\theta})^{-1/d}\}} |k_{\theta}(x,y)| dy &\le \sup_{x\in D} \int_{y\in D} |k_{\theta}(x,y)| dy=S^{-d}\|\Sigma^{(\theta)}
\|_{\infty}\\
&\le S^{-d}(1+\tau h_N (2K)^{d})\asymp S^{-d}\lesssim \|\CC_{\theta}\| \nu_m. \tag*{\qedhere}
\end{align*}
\end{enumerate}
\end{proof}

\begin{lemma}[Assouad, see e.g. Lemma 24.3 in \cite{van2000asymptotic}]\label{lemma:assouad}  Let $\Theta=\{0,1\}^\gamma$ and let $T$ be an estimator based on an observation from a distribution in the collection $\left\{P_\theta, \theta \in \Theta\right\}$. Then, for all $s>0,$
$$
\max _{\theta \in \Theta} 2^s \mathbb{E}_\theta d^s \bigl(T, \psi(\theta)\bigr) \geq \min _{H\left(\theta, \theta^{\prime}\right) \geq 1} \frac{d^s\bigl(\psi(\theta), \psi\left(\theta^{\prime}\right)\bigr)}{H\left(\theta, \theta^{\prime}\right)} \cdot \frac{\gamma}{2} \cdot \min _{H\left(\theta, \theta^{\prime}\right)=1}\bigl\|\mathbb{P}_\theta \wedge \mathbb{P}_{\theta^{\prime}}\bigr\| .
$$
\end{lemma} 


\begin{lemma}[Lower bound over $\F_2(r,\{\nu_m\})$]\label{lemma:lower_bound2}
Suppose $N> \log r$ and $r> m_*^d$. The minimax risk for estimating the covariance operator over $\F_2(r,\{\nu_m\})$ under the operator norm satisfies
\[
\inf_{\hat{\CC}}\sup_{\CC\in \F_2(r,\{\nu_m\})}\frac{\mathbb{E} \| \widehat{\mathcal{C}} - \mathcal{C} \|}{\|\mathcal{C} \| } \gtrsim \varepsilon_*.
\]
\end{lemma}

\begin{proof}[Proof of Lemma \ref{lemma:lower_bound2}]
We first notice that $\|\CC_{\theta}\|\asymp S^d$ for $\CC_{\theta}\in \F_2(r,\{\nu_m\})$, which leads to
\begin{align*}
\inf_{\hat{\CC}}\sup_{\CC\in \F_2(r,\{\nu_m\})} \frac{\mathbb{E} \| \widehat{\mathcal{C}} - \mathcal{C} \|}{\|\mathcal{C} \| }\asymp \inf_{\hat{\CC}}\sup_{\theta\in \{0,1\}^{\gamma_N} } S^d\, \E \|\hat{\CC}-\CC_{\theta}\|.
\end{align*}
Applying the lower bound reduction in Proposition~\ref{prop:reduction}, the inequality \eqref{eq:reduction1} in Proposition~\ref{prop:reduction} (with $M=S^d$) gives that
\[
\inf_{\hat{\CC}}\sup_{\theta\in \{0,1\}^{\gamma_N} } S^d\, \E \|\hat{\CC}-\CC_{\theta}\| \ge \inf_{\hat{\Sigma}}\sup_{\theta\in \{0,1\}^{\gamma_N}}  \E \|\hat{\Sigma}-\Sigma^{(\theta)}\|.
\]

Before we apply Assouad's Lemma \ref{lemma:assouad} to derive a lower bound for the covariance matrix estimation problem over the testing class $\{\Sigma^{(\theta)}:\theta\in \{0,1\}^{\gamma_N}\}$, we introduce some basic notation and definitions. Denote the joint distribution of $X_1,X_2, \ldots, X_N\iid \mcN(0,\Sigma^{(\theta)})$ by $\P_{\theta}$. For two probability measures $\P$ and $\mathbb{Q}$ with density $p$ and $q$ with respect to any common dominating measure $\mu$, write the total variation affinity $\|\P\land \mathbb{Q}\|=\int p\land q \, d\mu$. Let $H(\theta,\theta^{\prime})=\sum_{i=1}^{\gamma_N} |\theta_i-\theta^{\prime}_i|$ be the Hamming distance on $\{0,1\}^{\gamma_N}$.

Applying Assouad's Lemma \ref{lemma:assouad} with $d(\Sigma^{(\theta)},\Sigma^{(\theta^{\prime})})=\|\Sigma^{(\theta)}-\Sigma^{(\theta^{\prime})}\|, s=1, \gamma=\gamma_N$ gives that
\begin{align}\label{eq:assouad}
\inf_{\hat{\Sigma}}\sup_{\theta\in \{0,1\}^{\gamma_N}}  \E \|\hat{\Sigma}-\Sigma^{(\theta)}\| \ge \min_{H(\theta,\theta^{\prime})\ge 1}\frac{\|\Sigma^{(\theta)}-\Sigma^{(\theta^{\prime})}\|}{H(\theta,\theta^{\prime})}\cdot \frac{\gamma_N}{4}\cdot \min_{H(\theta,\theta^{\prime})=1}\|\P_{\theta}\land \P_{\theta^{\prime}}\|.
\end{align}
We shall prove the following bounds for the first and third factors on the right-hand side of \eqref{eq:assouad}:
\begin{enumerate}[label=(\alph*)]
    \item $\min_{H(\theta,\theta^{\prime})\ge 1}\frac{\|\Sigma^{(\theta)}-\Sigma^{(\theta^{\prime})}\|}{H(\theta,\theta^{\prime})}\gtrsim \frac{1}{\sqrt{N\gamma_N}}$; \label{eq:lowbound2_aux1}
    \item $\min_{H(\theta,\theta^{\prime})=1}\|\P_{\theta}\land \P_{\theta^{\prime}}\|\gtrsim 1$. \label{eq:lowbound2_aux2}
\end{enumerate}

\textbf{Proof of \ref{eq:lowbound2_aux1}.}
Let $L:=\left\{1\le \ell\le S^d: I_{\ell} \subseteq \left[\frac{K}{S},\frac{2K}{S}\right]^d \right\}$ and note that the cardinality of $L$ is $K^d$. We define a vector $v:=\sum_{\ell\in L} e_{\ell},$ where $\{e_\ell\}_{1\le \ell\le S^d}$ is the standard basis of $\R^{S^d}$ and $\omega=(\omega_i):=(\Sigma^{(\theta)}-\Sigma^{(\theta^{\prime})})v$. Note that there are exactly $H(\theta,\theta^{\prime})$ number of $\omega_i$ such that $|\omega_i|=\tau h_N K^d$, and $\|v\|=K^{d/2}$. This implies
\[
\|\Sigma^{(\theta)}-\Sigma^{(\theta^{\prime})}\|\ge \frac{\|(\Sigma^{(\theta)}-\Sigma^{(\theta^{\prime})})v\|}{\|v\|}=\frac{\sqrt{\|(\Sigma^{(\theta)}-\Sigma^{(\theta^{\prime})})v\|^2}}{\|v\|}= \frac{\sqrt{H(\theta,\theta^{\prime})(\tau h_N K^d)^2}}{K^{d/2}}.
\]
Recall that  $h_N=K^{-d}\sqrt{\frac{m_*^d}{N}}$ and $K=(m_{*}-1)/2\sqrt{d}$, then
\begin{align*}
&\min_{H(\theta,\theta^{\prime})\ge 1}\frac{\|\Sigma^{(\theta)}-\Sigma^{(\theta^{\prime})}\|}{H(\theta,\theta^{\prime})}\ge \min_{H(\theta,\theta^{\prime})\ge 1}\frac{\sqrt{H(\theta,\theta^{\prime})\Big(\tau K^{-d}\sqrt{\frac{m_*^d}{N}} K^d \Big)^2}}{H(\theta,\theta^{\prime})K^{d/2}}\\
& \hspace{3.25cm} =\min_{H(\theta,\theta^{\prime})\ge 1}\tau \sqrt{\frac{m_*^d}{N}} \frac{1}{\sqrt{H(\theta,\theta^{\prime})}K^{d/2}} 
=\tau \sqrt{\frac{m_*^d}{N}}\frac{1}{\sqrt{\gamma_N} K^{d/2}}
\asymp \frac{1}{\sqrt{N\gamma_N}}.
\end{align*}

\textbf{Proof of \ref{eq:lowbound2_aux2}.}
By \cite[Lemma 2.6]{tsybakov2008introduction}, the total variation affinity is lower bounded by
\[
\|\P_{\theta}\land \P_{\theta^{\prime}}\|\ge \frac{1}{2}\exp\biggl(-N \cdot D_{\mathrm{KL}}\Bigl(\mcN \bigl(0,\Sigma^{(\theta)}\bigr) \big\| \, \mcN \bigl(0,\Sigma^{(\theta^{\prime})}\bigr)\Bigr)\biggr).
\]
We will upper bound the Kullback-Leibler divergence using an explicit calculation. To that end,  note that for $\theta, \theta^{\prime}$ with $H(\theta,\theta^{\prime})=1$,
\begin{align}\label{eq:KLexplicit}
\begin{split}
\bigl\|(\Sigma^{(\theta^{\prime})})^{-1/2}\Sigma^{(\theta)}(\Sigma^{(\theta^{\prime})})^{-1/2}-I \bigr\| 
&\le \bigl\|(\Sigma^{(\theta^{\prime})})^{-1}\bigr\| \bigl\|\Sigma^{(\theta)}-\Sigma^{(\theta^{\prime})} \bigr\|\\ 
&=\frac{1}{\lambda_{\min} \bigl(\Sigma^{(\theta^{\prime})}\bigr)} \bigl\|\Sigma^{(\theta)}-\Sigma^{(\theta^{\prime})} \bigr\|\\
&\overset{\text{(i)}}{\le} \frac{4}{3} \bigl\|\Sigma^{(\theta)}-\Sigma^{(\theta^{\prime})} \bigr\|_{1}\overset{\text{(ii)}}{\le} \frac{4}{3}\tau h_N (2K)^d\overset{\text{(iii)}}{\le} \frac{1}{2},
\end{split}
\end{align}
where (i) follows by $\lambda_{\min} \bigl(\Sigma^{(\theta^{\prime})}\bigr)\ge \frac{3}{4}$ and $\|A\|\le \|A\|_1$ for any symmetric matrix $A$; (ii) follows by the fact there are exactly one nonzero row and column in the matrix $\Sigma^{(\theta)}-\Sigma^{(\theta^{\prime})}$, where the number of nonzero entries in that row/column is at most $(2K)^d$ and the absolute value of every nonzero entry is $\tau h_N$; (iii) follows from $\tau h_N (2K)^d\le \frac{1}{4},$ which we established while checking that $\F_2(r,\{\nu_m\})\subseteq \F(r,\{\nu_m\})$.

Denote by $\{\alpha_j\}_{j=1}^{S^d}$ the set of eigenvalues of the matrix $\bigl(\Sigma^{(\theta^{\prime})}\bigr)^{-1/2}\Sigma^{(\theta)} \bigl(\Sigma^{(\theta^{\prime})}\bigr)^{-1/2}-I$. It follows from \eqref{eq:KLexplicit} that $|\alpha_j|\le \frac{1}{2}$ for all $j$.
Hence, using the formula for the Kullback–Leibler divergence between two Gaussians (see e.g. \cite[Chapter 1]{pardo2018statistical}), we deduce that
\begin{align*}
&D_{\mathrm{KL}} \Bigl(\mcN \bigl(0,\Sigma^{(\theta)}\bigr) \big\| \, \mcN \bigl(0,\Sigma^{(\theta^{\prime})}\bigr) \Bigr)=\frac{1}{2}\left[\mathrm{Tr} \Bigl( \bigl(\Sigma^{(\theta^{\prime})} \bigr)^{-1}\Sigma^{(\theta)}-I \Bigr)-\log\frac{\det(\Sigma^{(\theta)})}{\det(\Sigma^{(\theta^{\prime})})}\right]\\
&=\frac{1}{2}\sum_{j=1}^{S^d} \bigl(\alpha_j-\log (1+\alpha_j)\bigr)\le \sum_{j=1}^{S^d} \alpha_j^2 
= \bigl\|(\Sigma^{(\theta^{\prime})})^{-1/2}\Sigma^{(\theta)}(\Sigma^{(\theta^{\prime})})^{-1/2}-I \bigr\|^2_{F}\\
&= \bigl\|(\Sigma^{(\theta^{\prime})})^{-1/2}\big(\Sigma^{(\theta)}-\Sigma^{(\theta^{\prime})}\big)(\Sigma^{(\theta^{\prime})})^{-1/2} \bigr\|^2_{F} \le \frac{16}{9} \bigl\|\Sigma^{(\theta)}-\Sigma^{(\theta^{\prime})} \bigr\|^2_{F},
\end{align*}
where in the first inequality we used that $x-\log(1+x)\le 2 x^2$ for $|x|\le \frac{1}{2}$, and the last inequality follows by the matrix norm inequality $\|ABC\|_{F}\le \|A\|\|B\|_{F}\|C^{\top}\|$ and $\lambda_{\min}(\Sigma^{(\theta^{\prime})})\ge \frac{3}{4}$.

Recall that $h_N=K^{-d}\sqrt{\frac{m_*^d}{N}}$ and $K=(m_{*}-1)/2\sqrt{d}$. For $\theta, \theta^{\prime}$ with $H(\theta,\theta^{\prime})=1$, there are exactly one nonzero row and column in the matrix $\Sigma^{(\theta)}-\Sigma^{(\theta^{\prime})}$, where the number of nonzero entries in that row/column is at most $(2K)^d$ and the absolute value of every nonzero entry is $\tau h_N$, so its squared Frobenius norm is bounded by 
\[
\bigl\|\Sigma^{(\theta)}-\Sigma^{(\theta^{\prime})} \bigr\|^2_{F}\le 2 \tau^2 h_N^2 (2K)^d\asymp \frac{K^{-d} m_*^{d}}{N} \asymp \frac{1}{N}.
\]
As a result,
\[
\bigl\|\P_{\theta}\land\P_{\theta^{\prime}} \bigr\|\ge \frac{1}{2}\exp\biggl(-N \cdot D_{\mathrm{KL}}\Bigl(\mcN \bigl(0,\Sigma^{\theta}\bigr) \big\| \, \mcN \bigl(0,\Sigma^{(\theta^{\prime})}\bigr)\Bigr)\biggr)\gtrsim 1,
\]
completing the proof of (b).

Combining \eqref{eq:assouad} with (a) and (b) gives that
\begin{align*}
\inf_{\hat{\Sigma}}\sup_{\theta\in \{0,1\}^{\gamma_N}}  \E \bigl\|\hat{\Sigma}-\Sigma^{(\theta)} \bigr\| &\ge \min_{H(\theta,\theta^{\prime})\ge 1}\frac{ \bigl\|\Sigma^{(\theta)}-\Sigma^{(\theta^{\prime})} \bigr\|}{H(\theta,\theta^{\prime})}\cdot \frac{\gamma_N}{4}\cdot \min_{H(\theta,\theta^{\prime})=1} \bigl\|\P_{\theta}\land \P_{\theta^{\prime}} \bigr\|\\
&\gtrsim \frac{1}{\sqrt{N\gamma_N}}\cdot \gamma_N = \sqrt{\frac{\gamma_N}{N}}\asymp \sqrt{\frac{m_*^d}{N}}\ge \varepsilon_*,
\end{align*}
where the last inequality follows from Lemma \ref{lemma:basic_relation} (C).  This completes the proof.
\end{proof}

\begin{lemma}[Lower bound over $\F_3(r,\{\nu_m\})$]\label{lemma:lower_bound3}
Suppose $r< m_*^d$. The minimax risk for estimating the covariance operator over $\F_3(r,\{\nu_m\})$ under the operator norm satisfies
\[
\inf_{\hat{\CC}}\sup_{\CC\in \F_3(r,\{\nu_m\})}\frac{\mathbb{E} \| \widehat{\mathcal{C}} - \mathcal{C} \|}{\|\mathcal{C} \| } \gtrsim \sqrt{\frac{r}{N}}.
\]
\end{lemma}

\begin{proof}[Proof of Lemma \ref{lemma:lower_bound3}]
 Following the same arguments as for $\F_2(r,\{\nu_m\})$, it is possible to check that $\F_3(r,\{\nu_m\})\subseteq \F(r,\{\nu_m\})$ and that the lower bound holds. We omit the proof for brevity.
\end{proof}

 \section{Proofs of results in Section \ref{sec:sparse}}\label{app:B}
    
\begin{proof}[Proof of Lemma \ref{lemma:Gamma1}]
We observe that
    \begin{align*}
    \|k\|_q^q \|k\|_{\infty}^{1-q}&=\sup_{x\in D}\left(\int_{D} |k(x,y)|^q dy\right) \|k\|^{1-q}_{\infty}
    =\sup_{x\in D}\left(\int_{D} \left(\frac{|k(x,y)|}{\|k\|_{\infty}}\right)^q dy\right) \|k\|_{\infty}\\
    &\overset{\text{(i)}}{\ge }\sup_{x\in D}\left(\int_{D} \frac{|k(x,y)|}{\|k\|_{\infty}} dy\right) \|k\|_{\infty}=\|k\|_1,
    \end{align*}
    where (i) follows since $q\in [0,1]$  and $|k(x,y)|\le \|k\|_{\infty}$. This implies $\Gamma_1(q,\CC)\ge \Gamma_1(1,\CC)$. The second inequality $\Gamma_1(1,\CC)\ge 1$ follows by \cite[Lemma B.1]{al2025covariancea}.
\end{proof}

\begin{proof}[Proof of Theorem \ref{thm:sparse_upper_bound}]
Set
\[
\rho := c_0 \frac{\sqrt{\|k\|_{\infty}}}{\sqrt{N}}\E\biggl[\sup_{x\in D} u(x)\biggr].
\]
By \cite[Theorem 2.2]{al2025covariancea},
\[
\begin{split}
\E \|\widehat{\CC}_{\widehat{\rho}}-\CC\| &\lesssim \|k\|_q^q \rho^{1-q}+\rho e^{-c N \big(\frac{\rho}{\|k\|_{\infty}}\big)\land \big(\frac{\rho}{\|k\|_{\infty}}\big)^2} \\
 &\lesssim \|k\|_q^q \|k\|_{\infty}^{1-q}\left(\frac{\Gamma_2(\CC)}{\sqrt{N} }\right)^{1-q}+\|k\|_{\infty}\frac{\Gamma_2(\CC)}{\sqrt{N}}e^{-c c_0 \Gamma^2_2(\CC)}\\
 &= \|\CC\| \left(\Gamma_{1}(q,\CC)\left(\frac{\Gamma_2(\CC)}{\sqrt{N} }\right)^{1-q}+\Gamma_1(0,\CC)e^{-cc_0\Gamma_2^2(\CC)}\frac{\Gamma_2(\CC)}{\sqrt{N}}\right)\\
 &\le \|\CC\| \left(\Gamma_{1}(q,\CC)\left(\frac{\Gamma_2(\CC)}{\sqrt{N} }\right)^{1-q}+\frac{\Gamma_2(\CC)}{\sqrt{N}}\right),
\end{split}
\]
where the last inequality follows by $\Gamma_1(0,\CC)e^{-cc_0\Gamma_2^2}\le \Gamma_1(0,\CC)e^{-C_0\Gamma_2^2}\le 1$. Therefore,
\begin{equation*}
    \frac{\E \|\widehat{\CC}_{\widehat{\rho}}-\CC\|}{\|\CC\|}
    \lesssim \Gamma_{1}(q,\CC)\left(\frac{\Gamma_2(\CC)}{\sqrt{N} }\right)^{1-q}+\frac{\Gamma_2(\CC)}{\sqrt{N}}\lesssim  \Gamma_{1}(q,\CC)\left(\frac{\Gamma_2(\CC)}{\sqrt{N} }\right)^{1-q}.  \qedhere
\end{equation*}
\end{proof}

\begin{lemma}\label{lemma:app sparsity lower bound inclusion}
    $\F_1(r,\eps_{N,r})\subseteq \F(\Gamma_1 (q),\Gamma_2).$
\end{lemma}

\begin{proof}[Proof of Lemma \ref{lemma:app sparsity lower bound inclusion}]
    \begin{enumerate}[label=(\alph*)]
    \item $\sup_{x\in D}k(x,x)=\sup_{1\leq j \leq r+1}\Sigma_{jj}=1.$
    \item For $\CC\in \F_1(r,\eps_{N,r}),$
    \begin{align*}
        \Gamma_1(q,\CC)&=\frac{\sup_{x\in D} \int_D |k(x,y)|^q dy}{\|\CC\|}
        =\frac{\max_{1\le i\le r+1}\sum_{j=1}^{r+1}|\Sigma_{ij}|^q \text{Vol}(I_{j})}{\frac{1}{r+1}\|\Sigma\|}\\
        &=\frac{\max_{1\leq i \leq r+1}\sum_{j=1}^{r+1}|\Sigma_{ij}|^q}{\|\Sigma\|} 
        \leq \frac{\max\{1,\frac{1}{2}(2\ell \eps_{N,r}^q+1)\}}{\|\Sigma\|}\\
        &\leq \frac{\max\{1,\frac{1}{2}(\Gamma_1(q)+1)\}}{\|\Sigma\|}\leq \frac{\Gamma_1(q)}{\|\Sigma\|}.
    \end{align*}
    Since $\Sigma(\theta) e_1=e_1$ for any $\theta\in \Theta$, we have that $\|\Sigma(\theta)\|\geq 1$. Consequently, it holds that
    $$
    \Gamma_1(q,\CC)\leq \Gamma_1(q).
    $$
    \item Since $\|\Sigma\|\leq \|\Sigma\|_1\leq \max\{1,\ell\eps_{N,r}+1/2\}\leq \max\{1,\frac{M\nu^{1-q}+1}{2}\}\leq 1$ for any $\CC\in \F_1(r,\eps_{N,r})$, we have that $\|\Sigma\|=1$ for any $\CC\in \F_1(r,\eps_{N,r}).$ It follows that
    \begin{align*}
        \Gamma_2(\CC)=\mathbb{E}_{u\sim \mathrm{GP}(0,\CC)}\biggl[\sup_{x\in D} u(x)\biggr]=\mathbb{E}_{\vec{u}\sim \mathcal{N}(0,\Sigma)}\biggl[\max_{1\leq i\leq r+1} \vec{u}_i\biggr]\leq \sqrt{2\log (r+1)}\leq \Gamma_2.
    \end{align*}
    \item For any $\CC\in \F_1(r,\eps_{N,r}),$ we have that \begin{align*}
        \Gamma_1(0,\CC)=\frac{\sup_{x\in D}\int_D|k(x,y)|^0dy}{\|\CC\|}\le\frac{\max_{1\leq i \leq r+1}\sum_{j=1}^{r+1}|\Sigma_{ij}|^0}{\|\Sigma\|}\leq \frac{r}{\|\Sigma\|}\leq r.
    \end{align*}
    Since $r\leq \exp(\frac{1}{2}\Gamma_2^2),$ we have that, for $C_0=\frac{1}{2},$
    $\Gamma_1(0,\CC)\exp\left(-C_0\Gamma_2^2 \right)\leq 1.$ \qedhere
\end{enumerate}
\end{proof}

   \begin{lemma}[Lower bound over $\F_1(r,\eps_{N,r})$] \label{lemma:approximate sparsity lower bound 2}
    Under the assumptions of Theorem \ref{thm:approximate sparsity lower bound}, the minimax risk for estimating the covariance operator over $\F_1(r,\eps_{N,r})$ under the operator norm satisfies 
    \begin{align*}
        \inf_{\hat{\CC}}\sup_{\CC\in \F_1(r,\eps_{N,r})}\frac{\mathbb{E}\|\hat{\CC}-\CC\|}{\|\CC\|} \gtrsim \Gamma_1(q)\left(\frac{\Gamma_2}{\sqrt{N}} \right)^{1-q}.
    \end{align*}
    \end{lemma}

 \begin{proof}[Proof of Lemma \ref{lemma:approximate sparsity lower bound 2}]
    Observe that 
    \begin{align*}
        \inf_{\hat{\CC}}\sup_{\CC\in \F_1}\frac{\mathbb{E}\|\hat{\CC}-\CC\|}{\|\CC\|} \hspace{-0.05cm} &\geq \inf_{\hat{\Sigma}}\max_{\Sigma\in \F_0}\frac{\mathbb{E}\|\hat{\Sigma}-\Sigma\|}{\|\Sigma\|} 
         \geq \inf_{\hat{\Sigma}}\max_{\Sigma\in \F_0}\|\hat{\Sigma}-\Sigma\|\\
         & \hspace{-0.75cm} = \inf_{\hat{\Sigma}_{r}}\max_{\Sigma\in \F_0} \biggl\|\begin{bmatrix}
            1 & \,\,\,\, 0_r^\top\\ \, 0_r &  \,\,\,\hat{\Sigma}_{r}
        \end{bmatrix}-\Sigma \biggr\|
         = \inf_{\hat{\Sigma}_{r}}\max_{\theta\in  \Theta}\frac{1}{2} \Bigl\|\hat{\Sigma}_{r}-I_{r}-\eps_{N,r}\sum_{m=1}^{r^*}\xi_mA_m(\lambda_m) \Bigr\|.\end{align*}
         From Lemma 3 in \cite{cai2012optimal},  we have that 
        \begin{align}\label{eq:cai approximate sparsity lower bound}
            \inf_{\hat{\Sigma}}\max_{\theta\in \Theta}\mathbb{E}\|\hat{\Sigma}-\Sigma(\theta)\|\geq \alpha r^* \min_{1\leq i \leq r^*}\|\bar{\mathbb{P}}_{i,0}\wedge \bar{\mathbb{P}}_{i,1}\|,
        \end{align}
        where $\alpha:=\min_{(\theta,\theta'):H(\xi(\theta),\xi(\theta'))\geq 1}\frac{\|\Sigma(\theta)-\Sigma(\theta')\|}{H(\xi(\theta),\xi(\theta'))}$, and $\bar{\mathbb{P}}_{i,a}=\frac{1}{2^{r^*-1}\text{Card}(\Lambda)}\sum_{\theta:\xi_i(\theta)=a}\mathbb{P}_{\theta}.$ Here, $H(\xi(\theta),\xi(\theta'))$ denotes the Hamming distance between the components of $\theta$ and $\theta'$ in $\Xi$. The distribution $\bar{\mathbb{P}}_{i,a}$ is the mixture distribution over all $\mathbb{P}_{\theta}=\mathcal{N}(0,\Sigma(\theta))$ with $\xi_i(\theta)$ fixed to be equal to $a$ and the other components of $\theta$ varying across all values in $\Theta$. From Lemma 6 in \cite{cai2012optimal}, under the assumptions that $1\leq N^{\beta}\leq \lfloor\exp (\frac{1}{2}\Gamma_2^2)\rfloor-1$ and $\Gamma_1(q)\leq MN^{(1-q)/2}\Gamma_2^{-3+q}$ along with our choice of $\nu$, it holds that
        \begin{align}\label{eq:cai mixture lower bound}
            \min_{1\leq i\leq r^*}\|\bar{\mathbb{P}}_{i,0}\wedge \bar{\mathbb{P}}_{i,1}\|\geq c_1,
        \end{align}
        for a universal constant $c_1>0.$ It remains to prove a lower bound for $\alpha.$ Let $v\in \R^r$ have entries $v_i=0$ for $1\leq i\leq r-r^*$ and $v_i=1$ for $r-r^*+1\leq i \leq r.$ Denoting $w=\bigl(\Sigma(\theta)-\Sigma(\theta')\bigr)v$, we have that $|w_i|=\ell\eps_{N,r}$ if $|\xi_i(\theta)-\xi_i(\theta')|=1.$ Since there are at least $H\bigl(\xi(\theta),\xi(\theta')\bigr)$ elements with $|w_i|=\ell\eps_{N,r}$, we have that
        $$
        \|\bigl(\Sigma(\theta)-\Sigma(\theta')\bigr)v\|\geq \sqrt{H\bigl(\xi(\theta),\xi(\theta')\bigr)} \, \ell\eps_{N,r}.
        $$
        If $H\bigl(\xi(\theta),\xi(\theta')\bigr)\geq 1,$ it follows that
        \begin{align*}
            \frac{\|\Sigma(\theta)-\Sigma(\theta')\|}{H\bigl(\xi(\theta),\xi(\theta')\bigr)}\geq \frac{\|\bigl(\Sigma(\theta)-\Sigma(\theta')\bigr)v\|}{\|v\| H\bigl(\xi(\theta),\xi(\theta')\bigr)}\geq \frac{\ell \eps_{N,r}}{\|v\| \sqrt{H\bigl(\xi(\theta),\xi(\theta')\bigr)}}\geq \frac{\ell \eps_{N,r}}{\|v\| \sqrt{r}}=\frac{\ell \eps_{N,r}}{r},
        \end{align*}
        since $H\bigl(\xi(\theta),\xi(\theta')\bigr)\leq r^*\leq r$ and $\|v\| \le\sqrt{r}$ by construction. Thus, we have shown that
        \begin{align}\label{eq:alpha lower bound}
            \alpha \geq \frac{\ell \eps_{N,r}}{r}.
        \end{align}
        Combining \eqref{eq:cai mixture lower bound} and \eqref{eq:alpha lower bound} with \eqref{eq:cai approximate sparsity lower bound} yields that
        \begin{align*}
                \inf_{\hat{\Sigma}}\max_{\theta\in \Theta}\mathbb{E}\|\hat{\Sigma}-\Sigma(\theta)\|\gtrsim \frac{r^*\ell \eps_{N,r}}{r}\gtrsim \ell \eps_{N,r} \gtrsim \Gamma_1(q)\left(\sqrt{\frac{\log r}{N}}\right)^{1-q}.
         \end{align*}
    To conclude, note that since $\Gamma_2\geq 2$, we have that  
    \begin{equation}
           \sqrt{\log r}=\sqrt{\log \left( \Bigl\lfloor \exp \Bigl(\frac{1}{2}\Gamma_2^2 \Bigr)\Bigr\rfloor-1\right)}\geq \frac{1}{2}\Gamma_2. \tag*{\qedhere}
    \end{equation}
    \end{proof}

\end{appendix}

\end{document}